%% file: ipm_inhom.tex
\documentclass[reqno, 11pt]{amsart}

\usepackage{siunitx}
\usepackage{placeins}
\usepackage{lineno}
\usepackage{amsaddr} 

\usepackage{multirow}
\usepackage{graphicx}
\usepackage{subcaption}

\usepackage[ruled,vlined,norelsize,noend]{algorithm2e}
\usepackage{tikz-cd}
\usepackage{pgfplots}
\usepackage{tikz}
\usepackage{pgf}
\usepgfplotslibrary{groupplots}
\usepackage[all,hyperref,numberbysection]{bi-discrete}
\numberwithin{equation}{section}
\numberwithin{figure}{section}
\numberwithin{table}{section}
\usepackage[backend=biber,maxbibnames=5,maxalphanames=5,style=authoryear,natbib=true, maxcitenames=1,mincitenames=1,bibencoding=utf8,giveninits,url=false,isbn=false,doi=true,eprint=true]{biblatex}
\let\cite\parencite
\renewbibmacro{in:}{}
\AtBeginBibliography{\small}
\DefineBibliographyStrings{english}{nodate = {in prep.}}

\providecommand{\dsig}{\, \mathrm{d}\sigma(x)}
\providecommand{\dx}{\, \mathrm{d}x}
\providecommand{\tria}{\mathcal{T}}
\providecommand{\tr}{\operatorname{tr}}

\def\N{\mathbb{N}}
\def\Lmeanfree{L_0^2}

\def\R{\mathbb{R}}

\def\Ra{\operatorname{Ra}}
\def\tol{\texttt{tol}}
\def\Hsim{H^1_{\sim}(\Omega)}
\def\Hdiv{H^1_{\divergence}(\Omega)}

\usepackage{cancel}
\usepackage{a4wide}
\definecolor{cobalt}{rgb}{0.0, 0.28, 0.67}
\definecolor{darkcerulean}{rgb}{0.03, 0.27, 0.49}
\hypersetup{citecolor=darkcerulean,linkcolor=darkcerulean,    urlcolor=darkcerulean}
\usepackage{enumitem}
\usepackage{listings}
\usepackage[thinlines]{easytable}
\usepackage[font=footnotesize,margin=0.2cm]{caption}
\usepackage{subcaption}
\usepackage{booktabs} 
\usepackage{rotating}

\makeatletter
\def\@settitle{\begin{center}\baselineskip14\p@
		\bfseries\Large\@title\end{center}} 
\makeatother

\makeatletter
\@namedef{subjclassname@2020}{%
	\textup{2020} Mathematics Subject Classification}
\makeatother

\makeatletter
\def\blx@err@patch#1{}
\makeatother

\begin{document}


\title[Pressure robustness for inhomogeneous Dirichlet conditions]{ Scott--Vogelius element and iterated penalty method for inhomogeneous Dirichlet boundary conditions
}

\author[Eickmann]{Franziska Eickmann}
\address[F.\ Eickmann]{Department of Mathematics, Technische Universit\"at Darmstadt, Dolivostraße 15, 64293 Darmstadt, Germany}
\email{eickmann@mathematik.tu-darmstadt.de}
\author[Scott]{L.\ Ridgway  Scott}
\address[L.~R.~Scott]{The University of Chicago (Emeritus), Chicago, Illinois 60637, United States}
\email{ridg@uchicago.edu}
\author[Tscherpel]{Tabea Tscherpel}
\address[T.\ Tscherpel]{Department of Mathematics, Technische Universit\"at Darmstadt, Dolivostraße 15, 64293 Darmstadt, Germany}
\email{tscherpel@mathematik.tu-darmstadt.de}

\subjclass[2020]{
	65N12,  
	65N15, 
	65N30, 
 	65F10,   
 	76D07,   
 	76M10
}

\keywords{inhomogeneous Dirichlet boundary conditions, quasi-optimality, Fortin operator,  Scott-Vogelius element, 
	exact divergence constraint, 
	pressure-robust, pressure robustness,
	Uzawa algorithm}

\date{\today}

\maketitle

\begin{abstract}We present quasi-optimal a priori error estimates for general mixed finite element methods to approximate solutions of the Stokes problem subject to inhomogeneous Dirichlet boundary conditions. 
For the Scott--Vogelius element this yields pressure-robust 
a priori error estimates. 
Due to the exact divergence constraint, this requires a compatibility condition for the boundary data to hold. 
A key tool is a modified Fortin operator \cite{EickmannTscherpel2025} capable of preserving this compatibility condition.  
Furthermore, we analyse the iterated penalty method, an Uzawa-type algorithm and we show its convergence and asymptotic pressure robustness. 
Numerical experiments support the theory and highlight the importance of the compatibility condition and the appropriate treatment of nearly singular vertices. 
\end{abstract}

\section{Introduction}

Accurate numerical simulation of incompressible fluid flow requires numerical methods that capture the essential features of the underlying fluid equations without introducing artefacts.
An important property in this regard is \emph{pressure robustness}, reflected in the fact that the discretisation error in the velocity is independent of the pressure. 
This property is critical for reliable solutions, particularly in low-viscosity regimes or in situations of low regularity of the pressure. 
In recent years various types of pressure-robust methods for incompressible fluid equations have been developed, see~\cite{John2017} for a review. 
One of them is the Scott--Vogelius mixed finite element method first introduced in~\cite{ScottVogelius1985}. 
Since the pressure space consists of the divergence of the velocity functions, it produces exactly divergence-free velocity approximations. 
This, in turn, leads to pressure-robust {quasi-optimal} error estimates for the Stokes problem, see~\cite[Ch.~12.3]{Brenner2008}. 
Here \emph{quasi-optimal} refers to the fact that the discretisation error is equivalent to the best approximation error in the chosen finite element spaces with the same norms on both sides of the estimate.  
This means that the estimate holds under minimal regularity assumptions. Such an estimate is called \emph{pressure-robust}, if the estimate on the velocity error does not include the pressure best approximation error.  
This stands in contrast to general mixed finite element methods, where quasi-optimality typically holds in a product norm for velocity and pressure. 
While the benefits of pressure-robust methods are fairly well-established, the majority of contributions focus on problems with homogeneous Dirichlet boundary conditions~\cite{Brenner2008, Ainsworth2022, Ainsworth2023}, and mixed boundary conditions, or assume that the inhomogeneous divergence constraint is not discretised, see~\cite{John2017}. 
However, in practical fluid flow simulations, inhomogeneous boundary conditions are highly relevant, for instance in channel flows, driven cavity flows, both with and without convection, i.e., for the Stokes and the Navier--Stokes equations. 

Surprisingly, the interplay between exactly divergence-free finite element methods and inhomogeneous boundary Dirichlet conditions has received limited attention, with~\cite{HeisterRebholzXiao2016} the only exception we are aware of.
For exactly divergence-free finite element methods, imposing inhomogeneous Dirichlet boundary conditions is more delicate, because the discrete boundary data are also required to respect a compatibility condition, in the sense that the mean normal trace vanishes. 
Failure to preserve this constraint  in the discretisation jeopardises the exact divergence constraint.

Thus, boundary data interpolation must be carefully designed, as in works~\cite{HeisterRebholzXiao2016} and \cite[Sec.~5.2]{BernardiGiraultRaviartHechtRiviere2024}.  
This significantly complicates the extension of classical  tools such as Fortin operators. 
For instance, the standard abstract Fortin lemma due to~\cite{Fortin1977}, see also~\cite[Sec.~26.2.3]{Ern2021b}, does not ensure  preservation of this compatibility condition. 

Even for general mixed finite element pairs, only recently the case of inhomogeneous Dirichlet boundary conditions has been considered, see~\cite{BernardiGiraultRaviartHechtRiviere2024,JessbergerKaltenbach2024}. 
For the purpose of well-posedness one may reduce the inhomogeneous Stokes problem to the homogeneous one by using a suitable extension operator. 
However, following this approach to prove quasi-optimal velocity error estimates in Sobolev seminorm gives rise to boundary error terms that cannot be expected to be eliminated easily, and hence may corrupt quasi-optimality.

In this work, we extend quasi-optimal error estimates for general mixed finite element methods to the setting of inhomogeneous Dirichlet boundary conditions. 
More specifically, velocity errors are  bounded in terms of the best approximation of both the velocity and pressure, without incurring boundary data error terms. 
This improves upon the a priori estimates for general mixed finite element pairs in~\cite[Thm.~5.2.4.
]{BernardiGiraultRaviartHechtRiviere2024}. Therein the authors present a similar a priori error estimate in the full Sobolev norm but with constrained best approximation error. 

A key ingredient in our proof is a Fortin operator preserving both the divergence in the dual of the discrete pressure space, as well as discrete traces, and zero mean normal traces. 
Such Fortin operators are constructed in~\cite{EickmannTscherpel2025} for a range of mixed finite element pairs. In most cases this requires a modification of the homogeneous Fortin operator. 
Also the numerical analysis in~\cite{JessbergerKaltenbach2024} is based on the assumption that a Fortin operator with similar properties exists.
Note that the Fortin operator serves as a theoretical tool and need not be implemented. 
In contrast, an interpolation operator that preserves zero-mean normal traces is required in the implementation due to the exact divergence constraints.
As a direct consequence of the existence of such a Fortin operator we obtain a pressure-robust error estimate for the Scott--Vogelius method applied to the Stokes problem subject to inhomogeneous Dirichlet boundary conditions. A local Fortin operator for the Scott--Vogelius element has been constructed in \cite{ParkerSüli2025} and \cite{EickmannGuzmanetal2025}.

Besides inf-sup stability a major challenge of the Scott--Vogelius element is the local characterisation of the pressure space. In 2D for homogeneous Dirichlet boundary conditions this has been established in~\cite{GuzmanScott2019}, and subsequently in~\cite{Ainsworth2022} for mixed boundary conditions, see Tab.~4 therein. 
In 2D the pressure space is known to satisfy linear constraints at singular vertices~\cite{Brenner2008, GuzmanScott2019}, which complicate the implementation.
Determining the {correct} pressure space for the Scott--Vogelius discretisation of the Stokes problem with inhomogeneous Dirichlet data depends on the objective. If only well-posedness (inf--sup stability) is needed, one may choose it as divergence of the homogeneous velocity space; if, however, an exact divergence constraint is required, then the pressure space should contain the divergence of all discrete velocity functions. 
In this work, we require both representations of the pressure space, which we therefore assume. Indeed a sufficient condition in 2D is that no boundary singular vertices are present in the triangulation.

Explicit knowledge of a basis of the pressure space is not needed, when applying an Uzawa-type iteration to approximate the discrete solutions of the Stokes problem in exactly divergence-free finite element spaces. 
It is based on iteratively solving smaller elliptic systems for the velocity and the pressure. 
In the special case of the Scott--Vogelius element, an Uzawa-type method with divergence penalisation is referred to as \emph{Iterated Penalty Method} (IPM) in the literature~\cite[Ch. 13.1]{Brenner2008}. 

General Uzawa algorithms are long known, see e.g.,~\cite{Glowinski1984,Fortin1983} and widely investigated, among others in~\cite{Nochetto2004}, and  the available results naturally apply to the IPM since it is a special case. 
However, leveraging  the exact divergence constraint for the IPM one may show stronger results. To date this has been analysed only for homogeneous Dirichlet boundary conditions in~\cite{Brenner2008} and for mixed boundary conditions in \cite{Ainsworth2023}. 

We revisit and extend the error analysis to inhomogeneous Dirichlet boundary conditions, proving that the IPM yields  asymptotically (in the iteration) pressure-robust approximations when the compatibility condition is enforced. 
Furthermore, we demonstrate that the IPM solution minimises a discrete energy functional and that the divergence error decays monotonically during the iteration process.

Our main contributions can be summarised as follows:
\begin{itemize}
\item  In Theorem~\ref{thm:quasiopt-mixed} we present error estimates for general mixed finite elements for the Stokes problem subject to inhomogeneous Dirichlet boundary conditions, which are quasi-optimal in norms on the product space of velocity and pressure. 
\item 
For the Scott--Vogelius method in Theorem~\ref{thm:SV-inhom} we establish a pressure-robust quasi-optimal estimate for the velocity, again for inhomogeneous Dirichlet boundary conditions. In the proof we use the fact, that the compatibility condition is a constraint, similarly as the divergence constraint, which can be handled by use of a suitable Fortin operator.
The same applies to other exactly divergence-free mixed finite element pairs, such as the Falk--Neilan element~\cite{FalkNeilan2013} and the Guzmán--Neilan element~\cite{GuzmanNeilan2014a,GuzmanNeilan2014}, based on suitable Fortin interpolation operators 
\item 
We analyse the Iterated Penalty Method for the Stokes problem subject to inhomogeneous Dirichlet boundary conditions. More specifically, in Theorem~\ref{thm:ipm-quasi-opt} we prove asymptotic pressure robustness and derive explicit convergence rates that depend on the inf-sup constant. 
\item 
In Theorem~\ref{thm:mon} we present a monotonicity result stating that the norm of the divergence of the IPM iterates decreases in the course of the iteration. 
\item 
We emphasise the role of compatibility conditions, see~\cite{HeisterRebholzXiao2016}, both in the analysis and in practice: 
without enforcing the zero mean normal trace, neither convergence of the IPM nor pressure robustness can be guaranteed. 
A local mesh refinement as in~\cite{Ainsworth2022} is used to remove boundary singular vertices as well as nearly singular vertices, which improves the inf-sup constant and hence also the convergence rate of the IPM. 
Those aspects are investigated by 2D numerical experiments in Sections~\ref{sec:NumExp_SV} and~\ref{sec:ipm-num-exp}. 
\end{itemize}
Note that due to the linearity of the Stokes problem numerical analysis as well as implementation may be reduced to the case of homogeneous Dirichlet boundary conditions, if there is a divergence-free extension of given (compatible) boundary data.  
In fact, the mixed method for the Scott--Vogelius element as well as the corresponding IPM represent a way to compute such a divergence-free extension in a stable manner.  
However, this reduction to homogeneous boundary conditions does not work for extensions to nonlinear equations such as Navier--Stokes or non-Newtonian fluid equations.
Thus, insights into the imposition of inhomogeneous boundary conditions are highly relevant to those applications. 

\subsubsection*{Outline} 
In Section~\ref{sec:mixed-method}, we consider the Stokes problem subject to inhomogeneous Dirichlet boundary conditions, and its discretisation with a general pair of inf-sup stable finite element spaces. 
We prove a quasi-optimal error estimate in Theorem~\ref{thm:quasiopt-mixed} and discuss for which finite element spaces the conditions are satisfied. 
In Section~\ref{sec:SV}, we recall the  Scott--Vogelius finite element spaces, and present the setup for inhomogeneous boundary conditions. 
We obtain a pressure-robust quasi-optimality estimate in Theorem~\ref{thm:SV-inhom}. Furthermore, we discuss practical aspects and present numerical experiments.  
Section~\ref{sec:ipm} introduces and analyses the IPM, starting from convergence. Then we present an asymptotically pressure-robust a priori error estimate in Theorem~\ref{thm:ipm-quasi-opt}, and finally a monotonicity result for the divergence norms in Theorem~\ref{thm:mon}. 
The section closes with numerical experiments illustrating the theoretical findings. An implementation can be found in Zenodo \cite{zenodo}.

\section{Mixed methods for inhomogeneous Dirichlet boundary conditions}\label{sec:mixed-method}

In this section we show a quasi-optimal error estimate for general mixed methods solving the Stokes equations subject to inhomogeneous boundary conditions. 
For this purpose we employ a Fortin operator for the inhomogeneous setting that preserves the divergence and also certain trace properties~\cite{EickmannTscherpel2025}. 
With such an operator, one may use the standard arguments to remove both constraints from the best approximation error, namely the divergence constraint and the discrete boundary condition.

\subsubsection*{Function spaces}
Let $\Omega\subseteq \R^d$ for $d\geq 2$ be a bounded (open) domain with Lipschitz boundary.
We denote by $(L^2(\Omega),\norm{\cdot}_{L^2(\Omega)})$  the standard Lebesgue space and its norm, and the subspace of functions with zero mean by
\begin{align} \label{eq:def-Lmeanfre}
	\Lmeanfree(\Omega) &\coloneqq \left\{f\in L^2(\Omega) \colon \dashint_\Omega f(x) \dx = 0\right\}. 
\end{align}
Here, the notation for the integral mean is $$\dashint_{\Omega} f \dx \coloneqq \abs{\Omega}^{-1} \int_{\Omega} f \dx.$$  
By $H^1(\Omega)$ we denote the standard Sobolev space with norm $\norm{\cdot}_{H^1(\Omega)}$. 
Furthermore, $H^1_0(\Omega)$ is the closure of the test functions $C^\infty_c(\Omega)$ with respect to the norm $\norm{\cdot}_{H^1(\Omega)}$, and by the Poincaré inequality the seminorm $v \mapsto \norm{\nabla v}_{L^2(\Omega)}$ is a norm on $H^1_0(\Omega)$. 
We denote its dual space by $H^{-1}(\Omega) \coloneqq (H^1_0(\Omega))'$.

In general, for a function space $X$ we denote by $X^d$ the corresponding space of vector-valued functions. 
On the boundary $\partial \Omega$ of $\Omega$ we consider the fractional Sobolev space $H^{1/2}(\partial \Omega)^d$, see~\cite{AdamsFournier2003,Mazya2011}. 
With this, the trace operator $\tr \colon H^1(\Omega)^d \to H^{1/2}(\partial \Omega)^d$ is bounded and invertible. 

We use the subspace of functions with zero mean normal traces and the subspace of divergence-free functions of $H^1(\Omega)$, denoted by 
\begin{align}
	\Hsim &\coloneqq \left\{v\in H^1(\Omega)^d\colon  \int_{\partial \Omega} \tr (v) \cdot n \dsig = 0\right\},\\
	\Hdiv &\coloneqq \big\{ v \in H^1(\Omega)^d \colon \divergence v = 0\big\}. 
\end{align}
Note that we have $\Hdiv \subset \Hsim$, thanks to the identity  
\begin{align}\label{eq:compatibility}
	\int_\Omega \divergence v \dx  = \int_{\partial \Omega } \tr( v)\cdot n \dsig.
\end{align}
Furthermore, \eqref{eq:compatibility} implies that 
\begin{align}
	\divergence(\Hsim) \subset L^2_0(\Omega).
\end{align} 

\subsection{The fluid equations}

For given $F\in H^{-1}(\Omega)^d$,  
$g \in \Hsim$
 and kinematic viscosity coefficient $\nu>0$, we want to find $u \in \Hsim$ and $p \in \Lmeanfree (\Omega)$ such that 
 \begin{subequations}\label{eq:Stokes}
	\begin{alignat}{3}
		- \nu \Delta u + \nabla p &= F \qquad &&\text{ in } \Omega,\\
		\divergence u &= 0 \qquad &&\text{ in } \Omega,
	\end{alignat}
subject to inhomogeneous Dirichlet boundary conditions
\begin{align}\label{eq:bc-inhom}
	\tr(u) = \tr(g) \qquad \text{ on } \partial \Omega. 
\end{align}
\end{subequations}
The requirement~$g \in \Hsim$ is needed for the compatibility with the fact that $\divergence (u) = 0$ and $\tr(u) = \tr(g)$ in view of~\eqref{eq:compatibility}. 
In fact, it would suffice to require that the trace of $g$ is in $H^{1/2}(\partial \Omega)^d $ and that its normal component is mean-free, since an extension is known to exist, cf. Lemma~\ref{lem:extension_op} below. 

Let us introduce the weak formulation of~\eqref{eq:Stokes}, where we denote by $\skp{F}{v}$ for both the integral $\int_{\Omega} F v \dx$ and the duality relation between $H^{-1}(\Omega)^d$ and $H^1_0(\Omega)^d$. 

\subsubsection*{Weak formulation} For $F\in H^{-1}(\Omega)^d$,  
$g \in \Hsim$  and $\nu >0$, find  $(u,p)\in \Hsim \times \Lmeanfree(\Omega)$ such that 
\begin{equation}\label{eq:Stokes-weak-inh}
	\begin{aligned}
		\nu\skp{\nabla u }{\nabla v} - \skp{p}{\divergence v}  &= \skp{F}{v} \qquad &&\text{for all  } v \in H_0^1(\Omega)^d,\\
		\skp{\divergence u}{q}  &= 0 \qquad &&\text{for all } q \in \Lmeanfree (\Omega),\\
		\tr(u)&= \tr(g) \qquad &&\text{on } \partial\Omega.
	\end{aligned}
\end{equation}

Existence and uniqueness follow from standard inf-sup theory by reformulating the problem to a problem with homogeneous boundary conditions. Specifically, we have \emph{inf-sup stability}~\cite{Girault1986} in the sense that there is a constant $\beta>0$ such that 
\begin{equation}\label{def:inf-sup-cond}
 \beta  \leq\inf_{q \in L^2_0(\Omega)\setminus \{0\}} \sup_{v\in H_0^1(\Omega)^d \setminus \{0\}} \frac{\skp{\divergence v}{q}}{\norm{\nabla v}_{L^2(\Omega)} \norm{q}_{L^2(\Omega)}}.
\end{equation}
This inf-sup condition is equivalent to the existence of a right-inverse of the divergence~\cite[Thm.~III.3.1]{Galdi2011}, referred to as \emph{\Bogovskii{} operator}. 

\begin{lemma}[\Bogovskii{} operator]\label{lem:Bog}
	There is a bounded linear operator $\mathcal{B} \colon L^2_0(\Omega) \to H^1_0(\Omega)^d$ such that
	\begin{align*}
		\begin{aligned}
			\divergence \mathcal{B} q  &= q \qquad &&\text{ for all } q \in L^2_0(\Omega), \\
			\norm{\nabla \mathcal{B} q}_{L^2(\Omega)} &\leq \beta^{-1} \norm{q}_{L^2(\Omega)} \qquad &&\text{ for all } q \in L^2_0(\Omega).
		\end{aligned}	
	\end{align*}
\end{lemma}

Using the \Bogovskii{} operator, a divergence-free extension operator can be constructed as used, e.g., in~\cite{DieningMalekSteinhauer2008} for a divergence corrected Lipschitz truncation. 

\begin{lemma}[{\cite[Lemma~1.4.20]{BernardiGiraultRaviartHechtRiviere2024}} divergence-free extension]\label{lem:extension_op}
		There is a bounded linear operator  $E\colon \tr(\Hsim)\to H_{\divergence}^1(\Omega)$ such that 
		\begin{align*}
			\tr(E \tau) = \tau \qquad \text{ for any }\tau \in \tr(\Hsim).
		\end{align*}  
		In particular there is a constant $c_E>0$ such that 
		\begin{align*}
			\norm{\nabla E\tau}_{L^2(\Omega)} \leq c_E \norm{\tau}_{H^{1/2}(\partial\Omega)} \qquad \text{for all }\tau \in \tr(\Hsim).
		\end{align*}
\end{lemma}

\begin{remark}\label{rmk:div-cor}
In~\eqref{eq:Stokes-weak-inh} we require that~$g \in \Hsim$. 
Hence, without loss of generality one can apply Lemma~\ref{lem:extension_op} and use $\tilde{g}\coloneq E(\tr(g))\in \Hdiv$ instead of $g$, since $\Hdiv\subset \Hsim$.
Then, reduction to the homogeneous problem for $ \overline{u}= u - \widetilde{g} \in H^1_{0,\divergence}(\Omega)$ yields the well-known existence to~\eqref{eq:Stokes-weak-inh}. 
Furthermore, one can show uniqueness, and the solution only depends on $\tr(g)$, and is independent of the extension. 
\end{remark}

\begin{remark}
The Cauchy stress is symmetric, and hence the viscous term actually reads $2 \nu \divergence (D u )$ instead of $ \nu \divergence \nabla u = \nu \Delta u$, for $Du \coloneqq \tfrac{1}{2}(\nabla u + (\nabla u)^\top)$ the symmetric velocity gradient. 
Note however, that for divergence-free functions both terms coincide, and hence it is equivalent to work with gradients instead of symmetric gradients at least for the weak solution and for finite element methods with exact divergence constraints. 
Since this is the main focus below, we restrict the presentation to gradients throughout. 
\end{remark}

\subsection{Mixed finite element methods}
We recall the setup for mixed finite element methods for incompressible fluid equations. 
We assume that the domain $\Omega$ has a polyhedral boundary and let $(\tria_i)_{i\in \mathbb{N}}$ be a family of conforming triangulations of $\overline{\Omega}$, see~\cite[Sec.~2.1]{Ciarlet2002} consisting of $d$-dimensional closed simplices. 
We define the mesh size $h_i \coloneqq \max_{T\in \tria_i} h_T>0$ where $h_T\coloneqq \diameter(T)$ for all $T\in \tria_i$. 
In the following, for simplicity we skip the index $i$. 
With slight abuse of notation, we use $h$ as index of a triangulation and refer to the sequence of triangulations as $(\tria_h)_{h}$, even though $h$ does not uniquely determine $\mathcal{T}_h$.
We denote by $\rho_T>0$ the radius of the largest ball inscribed in $T$, for $T \in \tria_h$. 
We assume that the family of triangulations $(\tria_h)_{h}$ is shape-regular, i.e., there is a constant $\chi>0$ such that 
\begin{align*}
	\rho_T \geq \chi h_T  \quad \text{ for all } T\in \tria_h \text{ and all } h .
\end{align*}
On each simplex $T\in\tria_h$ we consider finite-dimensional subspaces denoted with $X(T)\subseteq C(T)^d$ and $Q(T)\subseteq C(T)$, and define the conforming finite element spaces as
\begin{align}
	X_h\coloneqq X_h(\tria_h) &= \{v\in C(\overline \Omega)^d\colon v|_T \in X(T)~\forall\, T\in \tria_h\} \subseteq H^1(\Omega)^d
	\label{def:Xh} \\
	Q_h \coloneqq Q_h(\tria_h) &= \{q\in
	 \Lmeanfree(\Omega)\colon q|_T \in Q(T)~\forall\, T\,\in \tria_h\} \subseteq \Lmeanfree(\Omega).
	\label{def:Qh}
\end{align}
The finite element space with zero mean normal trace, and with discrete divergence constraint, respectively, are denoted by 
\begin{align}
	\label{def:Xhsim}
	X_{h,\sim}
	&\coloneqq 
	X_h \cap H_{\sim}^1(\Omega),\\
	X_{h,\divergence} 
	&\coloneqq 
	\{ v_h \in X_h\colon~\skp{\divergence v_h}{q_h} = 0~\forall q_h\in Q_h\}. 
	\label{def:Xhdiv}
\end{align}
Note that, for most examples one has $	X_{h,\divergence} \cap 	X_{h,\sim}  \not\subset \Hdiv$. 
The homogeneous subspaces are denoted by 
\begin{align}\label{def:Vh}
	V_{h}\coloneqq X_h \cap H_0^1(\Omega)^d  \quad \text{ and } \quad 
		V_{h,\divergence} \coloneqq X_{h,\divergence} \cap H^1_0(\Omega)^d. 
\end{align}
In the following, we shall use the function spaces 
\begin{align}\label{def:Lagrange-k}
	\mathcal{L}_k^1(\tria_h)^d \coloneqq \{v\in C(\overline \Omega)^d\colon v|_T \in \mathcal{P}_k(T)^d~\forall \,T\in \tria_h\},\\ \label{def:DG-ks}
	\mathcal{L}_k^0(\tria_h) \coloneqq \{v\in L^\infty(\Omega)\colon v|_T \in \mathcal{P}_k(T)~\forall \,T\in \tria_h\},	
\end{align}
where $\mathcal{P}_k(T)$ denotes the space of scalar polynomials of degree at most $k\in \N$ on a  simplex $T \in \mathcal{T}_h$. 
Note that $\mathcal{L}_k^1(\tria_h)^d$ arises as $X_h$ in~\eqref{def:Xh} by the choice $X(T) = \mathcal{P}_k(T)^d$ and $\mathcal{L}_k^0(\tria_h)\cap L_0^2(\Omega)$ as $Q_h$ in \eqref{def:Qh} by the choice $Q(T)=\mathcal{P}_k(T)$. 

In what follows we work with general inf-sup stable mixed finite element pairs. 

\begin{definition}{(discrete inf-sup condition)} \label{def:inf-sup}
	A pair of finite element spaces $V_h\times Q_h\subset H^1_0(\Omega)^d \times \Lmeanfree (\Omega)$ is called \emph{inf-sup stable} if there exists a constant $\overline \beta>0$ such that one has 
	\begin{equation}\label{def:discr-inf-sup-cond}
		\overline \beta \norm{q_h}_{L^2(\Omega)} \leq \sup_{v_h\in V_h \setminus \{0\}} \frac{\skp{\divergence v_h}{q_h}}{\norm{\nabla v_h}_{L^2(\Omega)}} \qquad \text{for all }q_h\in Q_h,
	\end{equation}
	uniformely in $h$.
\end{definition}

By the abstract Fortin Lemma, originally due to~\cite{Fortin1977}, the inf-sup stability implies existence of a Fortin operator in the homogeneous case.

\begin{remark}\label{rmk:FortinBogovskii}
\begin{enumerate}[label = (\alph*)]
\item \label{itm:homFortin} The discrete inf-sup stability of a pair of finite element spaces $V_h\times Q_h\subset H^1_0(\Omega)^d\times \Lmeanfree (\Omega)$ is equivalent to the existence of a Fortin operator $\overline{\Pi}_h \colon H_0^1(\Omega)^d\to V_h$~\cite[Lem.~26.9]{Ern2021b} that is divergence-preserving in the dual of the discrete pressure space and stable in the $H_0^1(\Omega)$ seminorm. 
\item \label{itm:discrBog}
We call a linear operator $\mathcal{B}_h \colon L^2_0(\Omega) \to V_h $ a discrete \Bogovskii{} operator, if 
\begin{alignat}{3}\label{eq:Bog-h}
\skp{q_h}{ \divergence \mathcal{B}_h (\xi)} 
&=  
\skp{q_h}{ \xi} \qquad 
&&\text{ for all } q_h \in Q_h \text{ and all } \xi \in L^2_0(\Omega),\\
\label{eq:Bog-h-stab}
\overline{\beta}
\norm{\nabla \mathcal{B}_h(\xi)}_{L^2(\Omega)} 
&\leq
  \norm{\xi}_{L^2(\Omega)} \quad &&\text{ for any } \xi \in L^2_0(\Omega).
\end{alignat} 
Its existence is a result of well-posedness of the following mixed problem:
 For any $\xi \in L^2_0(\Omega)$ find $w_h \eqqcolon \mathcal{B}_h(\xi) \in V_h$ such that 
\begin{equation}\label{eq:bog-mixed}
\begin{aligned}
	\skp{\nabla w_h}{\nabla v_h}&= 0  &&\text{ for any } v_h \in V_{h,\divergence},\\
	\skp{\divergence w_h}{q_h} &=  \skp{\xi}{q_h}  &&\text{ for any } q_h \in Q_{h}.
\end{aligned}
\end{equation}
Its stability follows from standard estimates for the mixed problem.
\end{enumerate}
\end{remark}

In the following, we consider $X_h$ and $Q_h$ such that the corresponding  $V_h\times Q_h\subset H^1_0(\Omega)^d \times \Lmeanfree (\Omega)$ is a pair of inf-sup stable finite element spaces. 

\begin{assumption}[FEM setting]\label{assumpt:FEM}
	Let $(X_h,Q_h)$ be a family of mixed finite element spaces corresponding to a sequence of shape regular triangulations $(\mathcal{T}_h)_{h}$ of $\Omega$ such that $V_h\times Q_h \subset H^1_0(\Omega)^d \times L^2_0(\Omega)$, see~\eqref{def:Vh}, is a family of inf-sup stable mixed spaces with constant $\overline\beta$ according to Definition~\ref{def:inf-sup}. 
\end{assumption}

Let $F \in H^{-1}(\Omega)^d$ and let $g_h \in X_{h} $ be a suitable approximation of $g \in \Hsim$.  
Then, the approximate problem to the inhomogeneous problem~\eqref{eq:Stokes-weak-inh} reads as follows. 

\subsubsection*{Discrete formulation}   Find $(u_h,p_h)\in X_h \times Q_h$ such that 
\begin{align}\label{eq:Stokes-mixed-inh}
	\begin{aligned}
		\nu \skp{\nabla u_h}{\nabla v_h} - \skp{p_h}{\divergence v_h} &= \skp{F}{v_h} \qquad &&\text{for all  } v_h \in V_h,\\
		\skp{\divergence u_h}{q_h} &=  0 \qquad &&\text{for all } q_h \in Q_h,\\
		\tr(u_h) &=\tr(g_h) \qquad && \text{on } \partial\Omega. 
	\end{aligned}
\end{align} 
Since the functions in $X_h$ are continuous, the trace is understood in the classical sense. 
Well-posedness of the discrete formulation is standard, see, e.g.,~\cite[Ch.~53]{Ern2021b}. 
Since we assume slightly more on the data later on for the proof of quasi-optimality, let us sketch the well-posedness result. 

\begin{lemma} Under Assumption~\ref{assumpt:FEM}
existence and uniqueness of~\eqref{eq:Stokes-mixed-inh} holds for any $F\in H^{-1}(\Omega)^d$ and any $g_h \in X_h$. 
\end{lemma}
\begin{proof}
For a given~$g_h \in X_{h}$ we define 
\begin{align}\label{def:ext-divfree-discr}
	\widetilde{g}_h \coloneqq g_h  -   \mathcal{B}_h \left(\divergence(g_h) - \dashint_{\Omega} \divergence (g_h) \dx \right) \in X_{h,\divergence}. 
\end{align}
Indeed, this is a function in $X_{h,\divergence}$, since for any $q_h \in Q_h \subset L^2_0(\Omega)$ by use of the properties of the homogeneous \Bogovskii{} operator in Lemma~\ref{lem:Bog} and the fact that $q_h \in L^2_0(\Omega)$ is orthogonal to constants we have that  
\begin{align*}
\skp{	\divergence(\widetilde{g}_h)}{ q_h } 
&= 
\skp{\divergence(g_h)}{ q_h }  - 
\left\langle\divergence\left( \mathcal{B}_h \left(\divergence(g_h) - \dashint_{\Omega} \divergence(g_h) \dx \right)\right), q_h \right\rangle \\
&= 
\skp{\divergence(g_h)}{ q_h }  - 
\left\langle \divergence(g_h) - \dashint_{\Omega} \divergence(g_h) \dx , q_h \right\rangle  = 0. 
\end{align*}
Then, we reduce the problem to the homogeneous one for $\overline{u}_h = u_h - \widetilde{g}_h \in V_{h,\divergence}$, as in~\eqref{def:Vh}. 
\end{proof}

\begin{remark}\label{rmk:gh-comp-wellp} 
At first glance one would think that the suitable approximation $g_h\in X_h$ of $g\in H^1_{\sim}$ has to be chosen such that $\tr(g_h) = \tr(v_h)$ for some $v_h \in X_{h,\divergence}$. However, let us argue that this is no additional restriction. 
Indeed, we may replace $g_h$ by any other function with the same trace, since the solution $(u_h,p_h)$ depends only on $\tr(g_h)$. Furthermore, for any inf-sup stable pair $(V_h,Q_h)$ the definition in~\eqref{def:ext-divfree-discr} provides a function $\widetilde{g}_h\in X_{h,\divergence}$ with $\tr(\widetilde{g}_h) = \tr({g}_h)$. 
This is not a contradiction to the fact that using exactly divergence-free finite element functions requires $g_h \in X_{h,\sim}$. 
Indeed, we shall see later that exact divergence constraints arise by both taking $Q_h = \divergence X_{h,\sim}$ as well as imposing that $u_h \in X_{h,\sim}\cap X_{h,\divergence} \subsetneq X_{h,\divergence}$. 
The latter is ensured by taking $g_h \in X_{h,\sim}$. 
\end{remark}

In the following, we employ a trace-preserving Fortin operator in order to derive quasi-optimal error estimates. 
The crucial property is the fact that it is defined on $H^1(\Omega)^d$, but it preserves discrete traces and zero mean normal trace.  

\begin{assumption}[Fortin operators]\label{assumpt:Fortin}
Under Assumption~\ref{assumpt:FEM} assume that for each $h$ there is an operator $\Pi_h \colon H^1(\Omega)^d \to X_h$ with the following properties:  
\begin{enumerate}[label = (\roman*)]
	\item (divergence preservation) 
	\label{itm:div-pres}
	For any $v \in H^1(\Omega)^d$ one has that 
	\begin{align*}
		\skp{\divergence \Pi_h v}{q_h} = \skp{\divergence v}{q_h}  \qquad \text{ for any } q_h\in Q_h; 
	\end{align*}
	\item (stability)
		\label{itm:stab}
	 There is a constant $c_F>0$ such that 
	\begin{align*}
	\norm{\nabla \Pi_h v}_{L^2(\Omega)} \leq c_F \norm{\nabla v}_{L^2(\Omega)} \qquad \text{ for any } v \in H^1(\Omega)^d \;\; \text{ uniformely in } h;
	\end{align*}
	\item (trace preservation) 
		\label{itm:trace}
	\begin{enumerate}[label = (iii\alph*)]
	\item \label{itm:trace-discr}
	For any $v \in H^1(\Omega)^d$ with discrete trace $\tr(v) \in \tr(X_h)$ one has $\tr(\Pi_h v) = \tr(v)$;
	\item \label{itm:trace-mean} $\Pi_h$ preserves zero mean normal traces in the sense that $\Pi_h(\Hsim) \subset X_{h,\sim}$.
	\end{enumerate}
\end{enumerate}
\end{assumption}
\begin{remark}
	For the proof of quasi-optimality for general mixed methods we only require \ref{itm:div-pres},\ref{itm:stab},\ref{itm:trace-discr}. Note however, that the construction of $\Pi_h$ satisfying \ref{itm:div-pres} in most cases ensures \ref{itm:trace-mean} as a by-product. In contrast, for exactly divergence free mixed methods we need \ref{itm:trace-mean} for the quasi-optimality and pressure robustness.
\end{remark}
Note that one may obtain a Fortin operator for a larger velocity space than $H^1_0(\Omega)^d$ from the abstract Fortin Lemma, see Remark \ref{rmk:FortinBogovskii}.
However, this does not ensure the trace preservation properties in Assumption~\ref{assumpt:Fortin}~\ref{itm:trace}.

\begin{remark}[Fortin operators] 
	\hfill
	\begin{enumerate}[label = (\alph*)]
	\item Note that the condition in Assumption~\ref{assumpt:Fortin}~\ref{itm:div-pres} implies that 
$\Pi_h(\Hdiv) \subset X_{h,\divergence}$. 
	\item Local Fortin operators $\overline{\Pi}_h \colon H^1_0(\Omega) \to V_h$ that  satisfy Assumption~\ref{assumpt:Fortin}~\ref{itm:div-pres} and~\ref{itm:stab} for any $v \in H^1_0(\Omega)^d$ are available for most mixed finite element pairs, for a review see~\cite[Sec. 2.2 and A.2.2]{Tscherpel2018}, including the following: 
	\begin{itemize} 
		\item $P2-P0$ element for $d = 2$, see~\cite[Sec.~8.4.3]{Boffi2013}.
		\item  Bernardi--Raugel element~\cite{BernardiRaugel1985} of first order for $d \in \{2,3\}$ and of second order for $d = 3$.
		\item conforming Crouzeix--Raviart element, for $d = 2$, and its generalisations in~\cite{CrouzeixRaviart1973,Mansfield1982}, and~\cite[Ex.~8.6.1, Prop. 8.6.2]{Boffi2013}, see~\cite{GiraultScott2003} for a Fortin operator.
		\item  MINI element for $d \in \{2,3\}$ introduced in~\cite{ArnoldBrezziFortin1984} for $d = 2$, see also~\cite[Sec.~8.4.2, 8.7.1]{Boffi2013}; for the Fortin operator see~\cite[]{BelenkiBerselliDieningEtAl2012}.
		\item Taylor--Hood element due to~\cite{TaylorHood1973} and its generalisations~\cite[Sec.~8.8.2]{Boffi2013} for $d \in \{2,3\}$; for a Fortin operator for $k \geq d$ we refer to~\cite[Sec.~3.1, 3.2]{GiraultScott2003} and the lowest order case $k = 2$ for $d \geq 3$ is covered in~\cite{DieningStornTscherpel2021}, see also~\cite{Scott2021}. 
		\item 
		Guzmán--Neilan elements for $d \in \{2,3\}$
		in~\cite{GuzmanNeilan2014a,GuzmanNeilan2014}. 
		\item 
		Scott--Vogelius element for $d = 2$ and $k \geq 4$~\cite{ScottZhang1990,GuzmanScott2019} on general meshes. 
		For a local Fortin operator see~\cite[A.2.2.2]{Tscherpel2018} based on~\cite{GuzmanScott2019}. 
		This approach uses local \Bogovskii{} operators on neighbouring patches, as developed in~\cite{DieningRuzickaSchumacher2010}. 
		\item Scott--Vogelius element on split meshes for $d \in \{2,3\}$~\cite{ArnoldQin1992,Zhang2005,GuzmanNeilan2018}, for a Fortin operator see~\cite[Sec.~4.3]{FuGuzmanNeilan2020} and~\cite{FarrellMitchellScottEtAl2022}.   
\end{itemize}
\item Existence of a Fortin operator satisfying Assumption~\ref{assumpt:Fortin} is 
immediate and natural for the Bernardi--Raugel element and the Guzmán-Neilan element, see~\cite{DieningStornTscherpel2025} for an exact implementation of the latter in 2D. 
For various mixed finite element pairs the construction of such a Fortin operator is presented  in~\cite{EickmannTscherpel2025}. For the Scott--Vogelius element such a local Fortin operator is constructed in \cite{ParkerSüli2025} (p-robustness) and \cite{EickmannGuzmanetal2025} for meshes including singular vertices.
\end{enumerate}
\end{remark}

\subsection{Quasi-optimality for the inhomogeneous problem}\label{sec:quasi-opt-general}
For the case of homogeneous boundary conditions, quasi-optimal error estimates are classical~\cite[Cor.~(12.5.18)]{Brenner2008}, see also~\cite{John2017} for a review. 
Here, we address the corresponding results for inhomogeneous boundary conditions, which improve on~\cite[Thm.~5.2.4]{BernardiGiraultRaviartHechtRiviere2024}.

\begin{theorem}[{quasi-optimality for mixed method}]\label{thm:quasiopt-mixed}
Let~$\nu >0$, let $F \in H^{-1}(\Omega)^d$ and let~$g\in H^{1}_\sim(\Omega)$ be given. 
Let $(u,p) \in \Hdiv \times \Lmeanfree(\Omega)$ be the weak solution to the Stokes problem~\eqref{eq:Stokes-weak-inh}.  
Under Assumption~\ref{assumpt:FEM} with inf-sup constant $\overline{\beta}$ and Assumption~\ref{assumpt:Fortin} on $\Pi_h$ let $g_h \in X_h$ be such that  $\tr(g_h)=\tr(\Pi_h g)$. 
Let $(u_h,p_h) \in X_h \times Q_h \subset H^1(\Omega)^d \times \Lmeanfree(\Omega)$ be the discrete solution to the mixed formulation~\eqref{eq:Stokes-mixed-inh}. 
Then we have that 
		\begin{align*}
			\norm{\nabla u - \nabla u_h}_{L^2(\Omega)} 
			&\leq 
			c_1 \inf_{v_h \in X_h} \norm{\nabla u - \nabla v_h}_{L^2(\Omega)} + \nu^{-1} \inf_{q_h \in Q_h}  	\norm{p - q_h}_{L^2(\Omega)}, 
			\\
			 \overline{\beta}\nu^{-1} \norm{ p - p_h}_{L^2(\Omega)} &\leq  
			c_1 \inf_{v_h \in X_h} \norm{\nabla u - \nabla v_h}_{L^2(\Omega)} 
			 + 
			\nu^{-1}(1 +  \overline{\beta}) \inf_{q_h \in Q_h}  	\norm{p - q_h}_{L^2(\Omega)},  
		\end{align*}
uniformly in $h$, where $c_1 =  2(1 + c_F)$ and $c_F$ is the constant in Assumption~\ref{assumpt:Fortin}~\ref{itm:stab}. 
\end{theorem}
\begin{proof} 
	By the standard arguments using Galerkin orthogonality we arrive at 
	\begin{align}\label{est:quasiopt-1}
		\norm{\nabla u - \nabla u_h}_{L^2(\Omega)}  \leq 2 \inf_{\substack{z_h \in X_{h,\divergence}\\ \tr(z_h)=\tr(g_h)}}  	\norm{\nabla u - \nabla z_h}_{L^2(\Omega)} + \nu^{-1} \inf_{q_h \in Q_h}  	\norm{p - q_h}_{L^2(\Omega)}. 
	\end{align} 
	To remove the constraints in the infimum, analogously to the homogeneous case~\cite[p.~345]{Brenner2008}, 
	see also~\cite[Lem.~4.3]{John2017},
	we employ the Fortin operator $\Pi_h \colon H^1(\Omega)^d \to X_h$ as in Assumption~\ref{assumpt:Fortin}. 
	The fact that it preserves both the divergence and discrete traces 
	allows us to show that for arbitrary $v_h \in X_h$ the function
		\begin{align*}
		z_h \coloneqq v_h + \Pi_h(u - v_h) 
	\end{align*}
	is in $X_{h,\divergence}$ and satisfies $\tr(z_h) = \tr(g_h)$. 
	Indeed, by Assumption~\ref{assumpt:Fortin}~\ref{itm:div-pres} the Fortin operator preserves the divergence, and in combination with the fact that the exact solution $u$ is divergence-free we obtain 
	\begin{align}\label{eq:zh_divfree}
		\skp{\divergence z_h}{q_h} &= \skp{\divergence v_h}{q_h} + \skp{\divergence \Pi_h (u-v_h)}{q_h} = \skp{\divergence v_h}{q_h} + \skp{\divergence (u-v_h)}{q_h} = 0
	\end{align}
	for any $q_h\in Q_h$. Therefore, we have $z_h \in X_{h,\divergence}$. 
	Since by Assumption~\ref{assumpt:Fortin}~\ref{itm:trace-discr} $\Pi_h$ preserves discrete traces, it follows that $\tr(v_h)  = \tr(\Pi_h(v_h))$ and $\tr(\Pi_h(u - g)) = 0$, and therefore we have
	\begin{equation}\label{eq:zh_trace}
	\begin{aligned}
		\tr(z_h) &= \tr(v_h) + \tr(\Pi_h u) - \tr(\Pi_h v_h) = \tr(v_h) + \tr(\Pi_h u) - \tr(v_h) = \tr(\Pi_h u) \\
		&= 
		\tr(\Pi_h g)
		= \tr(g_h).
	\end{aligned}
	\end{equation}
	With stability of the Fortin operator due to Assumption~\ref{assumpt:Fortin}~\ref{itm:stab} we obtain
	\begin{align*}
		 \norm{\nabla u - \nabla z_h}_{L^2(\Omega)} &
		\leq \norm{\nabla u - \nabla (v_h + \Pi_h(u - v_h))}_{L^2(\Omega)} \\
		&\leq \norm{\nabla u - \nabla v_h}_{L^2(\Omega)} + \norm{\nabla \Pi_h(u-v_h)}_{L^2(\Omega)} \\
		&\leq \norm{\nabla u - \nabla v_h}_{L^2(\Omega)} + c_F\norm{\nabla (u- v_h)}_{L^2(\Omega)} \\
		& = (1+c_F) \norm{\nabla u - \nabla v_h}_{L^2(\Omega)}.  
	\end{align*}
	Hence, employing that by~\eqref{eq:zh_divfree} and~\eqref{eq:zh_trace} we have $z_h\in X_{h,\divergence}$ and $\tr(z_h) = \tr(g_h)$ we find
	\begin{align*}
		\inf_{\substack{z_h\in X_{h,\divergence} \\ \tr(z_h) = g_h}} \norm{\nabla u - \nabla z_h}_{L^2(\Omega)} \leq (1+c_F)  \inf_{v_h\in X_h} \norm{\nabla u - \nabla v_h}_{L^2(\Omega)}.
	\end{align*}
	Applying this in \eqref{est:quasiopt-1} yields 
	\begin{align*}
		\norm{\nabla u - \nabla u_h}_{L^2(\Omega)} \leq 2(1+c_F) \inf_{v_h\in X_h} \norm{\nabla u - \nabla v_h}_{L^2(\Omega)}  + \nu^{-1} \inf_{q_h \in Q_h}  	\norm{p - q_h}_{L^2(\Omega)},  
	\end{align*}
	which proves the claim.	
	The pressure estimate can be obtained with standard arguments using the inf-sup stability, see, e.g.,~\cite[Lem.~4.4]{John2017}. 
\end{proof}

Let us discuss sufficient conditions for $g_h = \Pi_h(g)$ to be satisfied in Theorem~\ref{thm:quasiopt-mixed} with Fortin operator $\Pi_h$ from Assumption~\ref{assumpt:Fortin}. 
Since $\tr(g_h)$ enters the implementation it is worthwhile to relax the conditions as far as possible. 

\begin{assumption}[trace-preserving quasi-interpolation operator]\label{assumpt:corSZ}
	In the situation of Assumption~\ref{assumpt:FEM} assume that for each $h$ there is an operator $I_h \colon H^1(\Omega)^d \to X_h$ with the following properties:  
	\begin{enumerate}[label = (\roman*)]
		\item (trace preservation) 
		\label{itm:corSZ_trace}
		\begin{enumerate}[label = (i\alph*)]
			\item \label{itm:corSZ_trace-discr}
			For any $v \in H^1(\Omega)^d$ with discrete trace $\tr(v) \in \tr(X_h)$ one has $\tr(I_h v) = \tr(v)$;
			\item \label{itm:corSZ_trace-mean} $I_h$ satisfies 
			\begin{align*}
			\int_{\partial \Omega} (v - I_h v) \cdot n \dsig = 0 \qquad \text{ for any } v \in H^1(\Omega)^d,
			\end{align*}
			and in particular 
			 preserves zero mean normal traces  $I_h(\Hsim) \subset X_{h,\sim}$.
		\end{enumerate}
		\item (stability)
		\label{itm:corSZ_stab}
		There is a constant $c_I>0$ such that 
		\begin{align*}
			\norm{\nabla I_h v}_{L^2(\Omega)} \leq c_I \norm{\nabla v}_{L^2(\Omega)}  \qquad \text{ for any } v \in H^1(\Omega)^d \;\; \text{ and any } h.
		\end{align*}
	\end{enumerate}
\end{assumption}

\begin{lemma}\label{lem:reduction}
Let $(V_h,Q_h)_h$ be an inf-sup stable pair of spaces as in Assumption~\ref{assumpt:FEM}. 
Let $I_h$ be a trace-preserving quasi-interpolation operator as in Assumption~\ref{assumpt:corSZ}. 
Then there exists a family of Fortin operators $\Pi_h$ as in Assumption~\ref{assumpt:Fortin} with the property 
\begin{align*}
	\tr(\Pi_h(w)) = \tr(I_h(w)) \qquad \text{ for any } w \in H^1(\Omega)^d. 
\end{align*} 	In particular, in Theorem~\ref{thm:quasiopt-mixed} for given $g\in H^1_{\sim}(\Omega)$ it is sufficient to consider $g_h = I_h(g)$. 
\end{lemma}
\begin{proof}
A Fortin operator $\Pi_h$ satisfying Assumption~\ref{assumpt:Fortin} can be constructed as 
\begin{align*}
	\Pi_h v = I_h v + C_h^{\divergence} (v - I_h v) \qquad \text{ for any } v \in H^1(\Omega)^d.
\end{align*}
Here, $C_h^{\divergence} \colon \Hsim \to H^1_0(\Omega)^d$ is a divergence-correcting operator given by $C_h^{\divergence} = \overline{\Pi}_h \circ \mathcal{B} \circ \divergence $, with $\mathcal{B}$ the \Bogovskii{} operator, see Lemma~\ref{lem:Bog}, and $\overline{\Pi}_h$ the homogeneous Fortin operator, see Remark~\ref{rmk:FortinBogovskii}~\ref{itm:homFortin}.
Since $C_h^{\divergence}$ maps to functions with zero trace, it follows that 
\begin{align*}
	\tr(	\Pi_h v) = \tr(I_h v) \qquad \text{ for any } v\in H^1(\Omega)^d,
\end{align*}
which proves the claim. The remaining properties in Assumption \ref{assumpt:Fortin} are straightforward to check.
\end{proof}

\begin{remark}\label{rmk:cond-gh}
	\begin{enumerate}[label = (\alph*)]
		\item Note that with Assumption~\ref{assumpt:Fortin}~\ref{itm:trace-mean} and Assumption~\ref{assumpt:corSZ}~\ref{itm:corSZ_trace-mean} each of the conditions $\tr(g_h)=\tr(\Pi_h g)$ and $\tr(g_h) = \tr(I_h g)$ imply that $g_h \in \Hsim$. 
		\item If $g\in \Hsim$ satisfies that $g \in \tr(X_{h})$, then by the discrete trace preservation property in Assumption~\ref{assumpt:Fortin}~\ref{itm:trace-discr} we have $\tr(g)=\tr(\Pi_h g)$.
		For this reason, we may choose $g_h$ as an interpolation of $g$, and the assumption of the previous proposition is satisfied trivially.  
		\item It is also possible to obtain an operator as in Assumption~\ref{assumpt:corSZ} with the additional property that it is local.
	\end{enumerate} 
\end{remark} 

\begin{example}[trace-preserving interpolation operators] \label{rmk:Ih}\hfill
	\begin{enumerate}[label = (\alph*)]
		\item If $X_h$ is some (possibly enriched) Lagrange finite element space, the Scott--Zhang operator with discrete trace preservation can be modified to fulfil Assumption~\ref{assumpt:corSZ} while retaining its approximation properties. 
	For example, a modification of the degrees of freedom to include face moments yields \ref{itm:corSZ_trace-mean} in Assumption~\ref{assumpt:corSZ}. Alternatively a correction on each face by means of face bubble functions can be performed, see~\cite[Rmk. 5.4.10]{BernardiGiraultRaviartHechtRiviere2024}. 
	Both approaches are viable in the case $\mathcal{L}^1_d(\mathcal{T})^d \subset X_h$.
		\item In the following, we shall consider a simpler version of an operator $I_h$, that we use in the implementation below. 
		Here, a modification is done only on a single boundary face, see Section~\ref{sec:compat} below. 		
		\item 
		An operator satisfying Assumption~\ref{assumpt:corSZ} for continuous functions has been constructed  in~\cite{HeisterRebholzXiao2016}. 
		\item Another operator similarly as in Assumption~\ref{assumpt:corSZ} has been presented in~\cite[Ch.~6.2.1]{BernardiGiraultRaviartHechtRiviere2024}.
		\item \label{itm:Ih-split}
		On split meshes a modification of the Scott--Zhang operator can be obtained without requiring that $\mathcal{L}^1_d(\mathcal{T})^d \subset X_h$. 
		For example, for the Powell--Sabin type splits in 2D and 3D, each boundary face of the original mesh $\mathcal{T}_h$ has a new interior vertex in the split mesh $\widetilde{\mathcal{T}}_h$. Hence, one may use $\mathcal{L}^1_1(\widetilde{\mathcal{T}}_h)$ face `bubble' functions in the construction of $I_h$. Indeed the Scott--Zhang operator in $\mathcal{L}^1_1(\mathcal{T}_h)^d$ may be corrected with those functions. 
		
	\end{enumerate}
\end{example}
	
\section{The Scott--Vogelius element}\label{sec:SV}

Many of the standard mixed finite element methods are not pressure-robust. 
This has the effect that the velocity error depends on the pressure best approximation error, and on the inverse of the viscosity. 
This may cause severe problems in case of irregular pressure or large Reynolds numbers. 

Here we consider a particular example of pressure-robust finite elements, the Scott--Vogelius element~\cite{ScottVogelius1985, GuzmanScott2019}. 
We refer to~\cite{John2017} and \cite{Kreuzer2020,Kreuzer2021} for an incomplete list of alternative strategies to achieve pressure-robust quasi-optimal error estimates. 
The Scott--Vogelius finite element method leads to exactly divergence-free finite element functions for the velocity, and thus, ensures pressure robustness. 

\subsubsection*{The homogeneous case}

In the homogeneous case the velocity space is chosen as 
\begin{align}\label{def:Vh_SV}
V_h \coloneqq  \mathcal{L}_k^1(\tria_h)^d \cap H_0^1(\Omega)^d,
\end{align}
 and the pressure space is defined as the divergence of the velocity space
\begin{align}\label{def:Qh_SV_hom}
	Q_h \coloneqq \divergence V_h \subseteq \mathcal{L}^0_{k-1}(\tria_h) \cap L_0^2(\Omega).
\end{align}
Since the space of discretely divergence-free velocity functions is exactly divergence-free  
\begin{align*}
V_{h,\divergence}\coloneqq \{v_h\in V_h\colon \skp{\divergence v_h}{q_h} = 0~\forall q_h\in Q_h\} \subseteq \Hdiv \cap H_0^1(\Omega)^d
\end{align*}
one can show that the resulting method is pressure-robust as defined in~\cite{John2017}.

 As before, we work under the assumption of inf-sup stability of the family of pairs $(V_h,Q_h)_{h}$, see Assumption~\ref{assumpt:FEM}. 
While inf-sup stability is available for $k\geq 4$ in dimension $d = 2$~\cite{GuzmanScott2019} provided that there are no nearly singular vertices in the triangulation, in dimension $d = 3$ the situation is less well understood.  
Inf-sup stability is known to hold for $k \geq 6$ on certain regular triangulations~\cite{Zhang2011a}, also referred to as Freudenthal triangulation. 
Based on numerical experiments in~\cite{Farrell2024} 
Farrell et al.~conjecture inf-sup stability for $k\geq 4$ on Freudenthal meshes. 
By using split meshes~\cite{ArnoldQin1992,PowellSabin1977,WorseyFarin1987}, the polynomial degree for which inf-sup stability is available, can be lowered in both $d = 2$~\cite{Zhang2008,GuzmanLischkeNeilan2020,Fabien2022} and $d = 3$~\cite{Zhang2005,Zhang2011b,Fabien2022} dimensions; see~\cite{Neilan2020} for a review. However, the mesh split also causes an increase of the total dimension as discussed in~\cite{ScottTscherpel2024}.

In the following, we introduce the method for inhomogeneous boundary conditions, and prove the corresponding pressure-robust quasi-optimal error estimate in Subsection~\ref{sec:SV-mixed-inhom}. 
Then we comment on some practical aspects for the implementation including a mesh modification that improves the inf-sup constant, and the practical treatment of the compatibility condition on $g_h$ in Subsection~\ref{sec:practical_aspects}.  
 Finally, in Subsection~\ref{sec:NumExp_SV}, we present numerical examples showcasing the influence of the latter on the divergence of the discrete velocity approximation. Also the pressure robustness of the Scott--Vogelius method for inhomogeneous boundary conditions is visualised and contrasted with the non-pressure-robust Taylor--Hood method. 

\subsection{Scott--Vogelius method for inhomogeneous boundary conditions}\label{sec:SV-mixed-inhom}

For the Scott--Vogelius method we consider the following spaces for problem~\eqref{eq:Stokes-mixed-inh}
\begin{align}\label{def:Xh-SV}
	X_h &\coloneq \mathcal{L}_k^1(\tria_h)^d,\\
	X_{h,\sim} &\coloneq X_h \cap H_{\sim}^1(\Omega), 
	\label{def:Xhsim-SV}
	\end{align}
 see~\eqref{def:Xhsim}.
Then the homogeneous subspace in~\eqref{def:Vh_SV} is $V_h =  X_{h} \cap H^1_0(\Omega)^d$.
To avoid any ambiguity regarding the definition of the pressure space we make the following assumption. 

\begin{assumption}\label{assumpt:pressure-sp}
	Assume that we have $\divergence V_h = \divergence X_{h,\sim}$ for any $h$.
\end{assumption}
Thanks to Assumption~\ref{assumpt:pressure-sp} the pressure space for the homogeneous case coincides with the inhomogeneous one defined as 
\begin{align}
	\label{def:Qh_SV}
	Q_h& \coloneq \divergence X_{h,\sim}  = \divergence V_h \subset L^2_0(\Omega).
\end{align}
As demonstrated in~\cite[Table 3 Case I]{Ainsworth2022} in dimension $d = 2$  Assumption~\ref{assumpt:pressure-sp} is satisfied, provided that no triangulation $\mathcal{T}_h$ has any boundary singular vertices. 
For a definition of the latter we refer to~\cite[Sec.~2]{GuzmanScott2019} and Section \ref{sec:mesh_mod} below.

\begin{remark}\label{rmk:ex-div}
	For the uniqueness of the pressure to the inhomogeneous problem, inf-sup stability of $(V_h,\divergence V_h)$ suffices, and is available in~\cite{GuzmanScott2019}. 	
	On the other hand, to obtain that the discretely divergence-free velocity functions are exactly divergence-free, i.e., that
	\begin{align}\label{def:Xhdiv-SV}
		X_{h,\sim,\divergence} \coloneqq	\{ v_h \in X_{h,\sim}\colon~\skp{\divergence v_h}{q_h} = 0~\forall q_h\in Q_h\} \subset \Hdiv,
	\end{align}
	the pressure space $Q_h = \divergence X_{h,\sim}$ has to be used. 
	Let us remark that the latter is essential to set up the iterated penalty method in the subsequent section. 
	
	Related to this, let us recall that well-posedness of the discrete problem does not require $g_h \in X_h$ to be compatible in the sense that $g_h \in X_{h,\sim}$. 
	If we do not ensure that $g_h \in X_{h,\sim}$, then the approximate solutions may not be exactly divergence-free. 
	However, under Assumption~\ref{assumpt:pressure-sp}, if $g_h \in X_{h,\sim}$, then the approximate solution $u_h$ is exactly divergence-free.
\end{remark}

In the following, we obtain a pressure-robust version of quasi-optimality of the Scott--Vogelius method applied to the inhomogeneous problem~\eqref{eq:Stokes-mixed-inh}. 
This is a refined version of Theorem~\ref{thm:quasiopt-mixed}. 

As before this is based on a trace-preserving Fortin operator, see Assumption~\ref{assumpt:Fortin}. 

\begin{proposition}[\cite{EickmannTscherpel2025}]\label{prop:SV-op}
For $k \in \mathbb{N}$ let $X_h, X_{h,\sim,\divergence},V_h, Q_h$ as in~\eqref{def:Xh-SV}, \eqref{def:Xhdiv-SV}, \eqref{def:Vh_SV} and~\eqref{def:Qh_SV} be the discrete spaces of the Scott--Vogelius element and let Assumptions~\ref{assumpt:FEM} and~\ref{assumpt:pressure-sp} be satisfied. 
\begin{enumerate}[label = (\roman*)]
	\item	If $k \geq d$, then there exists an operator $I_h$ satisfying Assumption~\ref{assumpt:corSZ}. 
	\item There exists a Fortin operator $\Pi_h$ as in Assumption~\ref{assumpt:Fortin} with the property that 
	\begin{align*}
		\tr(\Pi_h v) = \tr(I_h v) \qquad \text{ for all } v \in H^1(\Omega)^d.
	\end{align*}
\end{enumerate}
\end{proposition}
Note that for general meshes in $d = 2$ Assumption~\ref{assumpt:FEM} is satisfied, provided that $k \geq 4$, and hence requiring $k \geq d=2$ is no additional restriction. See also \cite{EickmannGuzmanetal2025} for a local Fortin operator in $d=2$ if $k\geq 4$.

\begin{theorem}[quasi-optimality  for inhomogeneous boundary data]\label{thm:SV-inhom}
	Let~$\nu >0$, let $F \in H^{-1}(\Omega)^d$ and let~$g\in H^{1}_\sim(\Omega)$ be given. 
	Let $(u,p) \in \Hdiv \times \Lmeanfree(\Omega)$ be the weak solution to the Stokes problem~\eqref{eq:Stokes-weak-inh}.  
For some $k \in \mathbb{N}$ let $X_h, X_{h,\sim,\divergence},V_h, Q_h$ be as in~\eqref{def:Xh-SV}, \eqref{def:Xhdiv-SV}, \eqref{def:Vh_SV} and~\eqref{def:Qh_SV} satisfying Assumptions~\ref{assumpt:FEM} with inf-sup constant $\overline \beta>0$. Additionally let Assumption~\ref{assumpt:pressure-sp} be satisfied.
Furthermore, let $g_h \in X_h$ be such that  $\tr(g_h)=\tr(I_h g)= \tr(\Pi_h g)$, for the operators as in Proposition~\ref{prop:SV-op}, i.e., in  particular we have $g_h \in X_{h,\sim}$. 
Let $(u_h,p_h) \in X_{h,\sim,\divergence} \times Q_h \subset \Hdiv \times \Lmeanfree(\Omega)$ be the discrete solution to~\eqref{eq:Stokes-mixed-inh}.  
Then we have that 
	\begin{align*}
		\norm{\nabla u - \nabla u_h}_{L^2(\Omega)} &\leq c_1 \inf_{v_h \in X_h} \norm{\nabla u - \nabla v_h}_{L^2(\Omega)} , 
		\\
		\overline{\beta}\nu^{-1} \norm{ p - p_h}_{L^2(\Omega)} &\leq  
		c_1 \inf_{v_h \in X_h} \norm{\nabla u - \nabla v_h}_{L^2(\Omega)} 
		+ 
	\overline{\beta}	\nu^{-1}  \inf_{q_h \in Q_h}  	\norm{p - q_h}_{L^2(\Omega)},  
	\end{align*}
	uniformly in $h$, where $c_1 = 2(1 + c_F)$ and $c_F$ is the constant in Assumption~\ref{assumpt:Fortin}~\ref{itm:stab}, see Proposition~\ref{prop:SV-op}. 
\end{theorem}
\begin{proof}
	The proof follows analogously to the one of Theorem~\ref{thm:quasiopt-mixed}. The pressure term vanishes due to the exact divergence constraint. Using Galerkin orthogonality we arrive at  
	\begin{align*}
		\norm{\nabla u - \nabla u_h}_{L^2(\Omega)}  \leq 2 \inf_{\substack{z_h \in X_{h,\sim,\divergence}\\ \tr(z_h)=\tr(g_h)}}  	\norm{\nabla u - \nabla z_h}_{L^2(\Omega)}.
	\end{align*} 
	As above we employ the Fortin operator to remove the constraints in the infimum which proves the claim. 
	The pressure estimate is obtained with standard arguments involving the inf-sup stability. 
\end{proof}

Note that the estimate for the velocity is independent of the pressure, i.e., it is pressure-robust.

\subsection{Practical aspects}\label{sec:practical_aspects}

We now consider the Stokes problem in  $d = 2$ dimensions. 
In the following, we investigate two aspects relevant for the successful implementation of the Scott--Vogelius method. 

First, in Subsection~\ref{sec:mesh_mod} we employ a local mesh modification as proposed in~\cite{Ainsworth2022}.  The objective is to ensure Assumption~\ref{assumpt:pressure-sp} and to improve the inf-sup constant. 
The second one in Section~\ref{sec:compat} addresses the question of enforcing the compatibility condition on the boundary data. This is achieved similarly as in~\cite{HeisterRebholzXiao2016}. 
An alternative approach to handle problems with inf-sup stability by means of postprocessing is presented in \cite{Park2020}. 

\subsubsection{Local mesh modification}\label{sec:mesh_mod}
 
 To ensure that Assumption~\ref{assumpt:pressure-sp} is satisfied, for the numerical implementation we modify the triangulations such that they do not contain any singular vertices. Then, Assumption~\ref{assumpt:pressure-sp} is indeed satisfied.  
 Furthermore, it is known that the inf-sup constant is very small, if nearly singular vertices exist in the triangulation. 
 This can also be circumvented by a local mesh modification. 
 Note that this is already of interest for homogeneous Dirichlet boundary conditions. 
 Such local mesh modifications have been suggested in~\cite[Rmk.~2]{Ainsworth2022}. 
 
 To make the above precise, let us recall the notion of singular vertices in a triangulation, see~\cite{GuzmanScott2019}. For a vertex $z$ in a triangulation $\mathcal{T}$ let $T_1, \ldots, T_N \in \mathcal{T}$ be its $N\in\N$ adjacent triangles. If $z$ is a boundary vertex, we assume $T_1$ and $T_N$ to have a boundary edge. We enumerate the triangles such that $T_i$ and  $T_{i+1}$ share an edge for each $i \in {1, \ldots, N-1}$ and that $T_1$ and $T_N$ share an edge if $z$ is an interior vertex. 
 Let $\theta_i \in (0,2\pi)$ denote the interior angles at $z$ in $T_i$, for each $i \in \{1, \ldots, N\}$ and we set 
 \begin{align*}
 	\theta(z)&\coloneqq \begin{cases}
 		\max \{|\sin(\theta_1 + \theta_2)|, \dots, |\sin(\theta_{N-1} + \theta_N)|, |\sin(\theta_N + \theta_1)|\} & \text{if } z\in\Omega,\\
 		\max \{|\sin(\theta_1 + \theta_2)|, \dots, |\sin(\theta_{N-1} + \theta_N)|\} & \text{if }z\in\partial\Omega.
 	\end{cases}
 \end{align*}
 Then the set of \emph{singular vertices} in $\mathcal{T}$ is defined as
 \begin{align*}
 	\mathcal{S}(\mathcal{T}) &\coloneqq  \{ z
 	\in \mathcal{V}(\mathcal{T})\colon \theta(z) = 0\},
 \end{align*}
 where $\mathcal{V}(\mathcal{T})$ denotes the set of all vertices in $\mathcal{T}$. 
 Note that for interior singular vertices the adjacent edges are contained in exactly two straight lines. For boundary singular vertices the same is true, if all adjacent edges are extended to $\Omega^c$. 
 
Note that the inf-sup constant is affected by  $\theta_{\min}\coloneqq \min_{z\not\in \mathcal{S}(\mathcal{T})} \theta(z)$, and vertices $z$ with small $\theta(z)$ are referred to as \emph{nearly singular vertices}. 
It has been proved in~\cite{GraessleBohneSauter2024} that the inf-sup constant is small if $\theta_{\min}$ is small. When setting the parameter $\eta=0$ in their method it reduces to the Scott--Vogelius method. More precisely their Theorems 2 and 6 ensure the following equivalence
\begin{align*}
	c\theta_{\min} \leq \overline{\beta} \leq C \theta_{\min},
\end{align*}
where $c, C>0$ depend exclusively on the shape-regularity of the mesh and on $\Omega$.

The idea of the local mesh modification is simple but very effective. The triangulation is searched for \emph{possibly singular vertices}, i.e., 
\begin{enumerate}[label = (\alph*)]
	\item  interior vertices if they have $4$ adjacent triangles and 
	\item boundary vertices $z$, if they have
	\begin{align*}
		\begin{cases}
\text{$\leq2$ adjacent triangles} \quad \quad & \text{if $\Omega$ is convex at $z$},\\
 \text{$\leq3$ adjacent triangles} & \text{if $\Omega$ is not convex at $z$. }
		\end{cases}·
	\end{align*} 
\end{enumerate} 
Evidently, this is a larger set of vertices than the set of (nearly) singular vertices, but it allows us to avoid investigating all angles in the triangulation. 

In a second step we refine the triangulation locally in neighbourhoods of the possibly singular vertices: 
As in~\cite{Ainsworth2022} in both cases (a) and (b) we employ a local barycentric split of the adjacent triangles. 
This makes the respective vertex non-(nearly)-singular, and does not generate new (nearly) singular vertices. 
Starting from an original mesh $\tria_h$ this modified mesh will be denoted by $\tilde{\tria}_h$.

\subsubsection{Compatibility}\label{sec:compat}

Recall that $g \in \Hsim$ is given and that it is necessary that $g_h \in X_{h,\sim}$ in order to have an exactly divergence-free velocity solution $u_h \in X_{h,\sim,\divergence} \subset \Hdiv$. 
Specifically, for $g_h$ we need to ensure that
\begin{align}\label{eq:comp-gh}
	\int_{\partial \Omega} \tr(g_h) \cdot n \dsig= 0.
\end{align} 
Already in~\cite{HeisterRebholzXiao2016} the necessity of this constraint has been stated and a way to achieve it has been suggested.
Therein, a nodal interpolation operator mapping to $\tr(\mathcal{L}^1_{k}(\mathcal{T}_h)^d)$ is modified on each edge locally, in a way that all  nodal values except of one edge interior nodal value is preserved. 
One edge interior nodal value is replaced by preserving only the tangential component of the nodal value as well as the mean normal trace over the respective edge. 
 
\subsubsection*{Modification on a single edge} 

To obtain an exactly divergence-free approximation $u_h$ to the Stokes equation, we require that $g_h \in X_{h,\sim}$. 
To ensure this in the implementation for given data $g \in \Hsim\cap C(\overline{\Omega})^d$ in the Stokes problem~\eqref{eq:Stokes-weak-inh} we choose 
\begin{align}\label{eq:gh-choice}
	g_h = \tilde I_h g,
\end{align}
where $\tilde I_h \colon C(\overline \Omega)^d \to X_{h}$ satisfies~\ref{itm:corSZ_trace}--\ref{itm:corSZ_stab} in Assumption~\ref{assumpt:corSZ}. 
Let $L_h \colon C(\overline \Omega)^d \to X_{h}$ be the Lagrange interpolation operator applied componentwise. 
 Furthermore, we choose a single boundary face $f$ of $\mathcal{T}_h$ and we denote by $T_f$ the simplex in $\mathcal{T}_h$ such that $f \subset T_f$. 
 Let $b_f \in \mathcal{P}_{d}(\mathcal{T})$ be the normalised scalar  boundary face bubble with $\support(b_f) \subset T_f$ and let $n_f$ denote the outer (with respect to $\partial\Omega$) unit normal vector on $f$. 
 Then, for any $v \in C(\overline \Omega)^d$ we define  
 \begin{align}\label{def:tilde-Ih}
 	\tilde I_h v =  L_h v - c_f b_f n_f \qquad \text{ with }   c_f \coloneqq \left(\int_{f} b_{f} \dsig\right)^{-1} \int_{\partial\Omega} (L_h v  - v) \cdot n \dsig. 
 \end{align}
It is straightforward to check that $\tilde I_h v \in X_{h,\sim}$, if $v \in \Hsim \cap C(\overline{\Omega})^d$. 
In this case $c_f$ simplifies to 
\begin{align}
	  c_f \coloneqq \left(\int_{f} b_{f} \dsig\right)^{-1} \int_{\partial\Omega} L_h v  \cdot n \dsig
\end{align}
and $\tilde I_h v$ can be computed without needing to integrate $v$ exactly. 
Furthermore, with standard arguments one can show that 

\begin{align*}
	\norm{c_f b_f n_f}_{L^2(\Omega)} \lesssim \norm{v - L_h v}_{L^2(\Omega)} + h_T \norm{\nabla (v-L_h v)}_{L^2(\Omega)},
\end{align*}
for any $v\in C(\overline{\Omega})^d$. 
Then, by the stability and approximation property of $L_h$ the claim follows. 

The only difference from $I_h$ satisfying  Assumption~\ref{assumpt:corSZ} is the fact that $\tilde I_h$ is not defined on $H^1(\Omega)^d$, but on $C(\overline \Omega)^d$. 
In comparison to Proposition~\ref{prop:SV-op} here an interpolation operator is used instead of a quasi-interpolation operator, and the correction is done on a single face. 
In the numerical examples below we choose $g_h$ as in \eqref{eq:gh-choice}. 

\subsection{Numerical experiments for Scott--Vogelius method}\label{sec:NumExp_SV}
We now present some numerical tests in 2D performed with the legacy version \texttt{2019.2.0.dev0} of  FEniCS~\cite{AlnaesBlechtaHakeEtAl2015,LoggMardalWells2012}. 
For the solution of linear systems the solver mumps~\cite{Amestoy2001} is used and the plots are produced using matplotlib~\cite{Hunter2007}. The full code is available at Zenodo~\cite{zenodo}.

\subsubsection*{Manufactured solution}

For $\psi(x) =\sin(4\pi |x|^2)$ for $x \in \setR^2$, and for a parameter  $\Ra>0$ we consider the divergence-free function $u$ and the pressure function $p$ given by
\begin{align}\label{eq:exact_sol} 
	u(x)& = \curl \psi(x) = 8\pi \cos(4\pi|x|^2) \begin{pmatrix} x_2 \\ - x_1 \end{pmatrix}  \quad \text{ and } \quad 
	p(x) =10 \sin(\tfrac{\pi}{40}x_1)\sin(\tfrac{\pi}{20}x_2) \Ra.
\end{align}
Here $\curl$ denotes the 2D vector curl operator $\operatorname{curl}\psi=\begin{pmatrix}
	\partial_{x_2} \psi , -\partial_{x_1} \psi\end{pmatrix}^\top$. As domain we choose $\Omega =  (0,1)^2$.
The pair of functions $(u,p)$ forms a weak solution to the inhomogeneous Stokes problem~\eqref{eq:Stokes-weak-inh} for $\nu = 1$ and  $F \coloneqq -\nu\Delta u + \nabla p$.
The boundary data is chosen as $g\coloneqq u\in \Hsim\cap C(\overline{\Omega})^2$ where only $\tr(g) = \tr(u)$ on $\partial\Omega$ enters the problem and the numerical method. 
The choices of $\mathrm{Ra} $ are specified in the numerical experiments below.

\subsubsection*{Discretisation} 
We choose quasiuniform triangulations with approximately $N\in \{4, 8, 16, 32,$\\$64,128\}$ triangles per direction constructed with the FEniCS built-in mesh generator\\ \texttt{generate\_mesh(domain, N)}.
Then, the meshes are modified as described in Section~\ref{sec:mesh_mod}, which ensures that there are no singular vertices in the mesh. In this case thanks to~\cite[p.~517]{GuzmanScott2019} we have
\begin{align}\label{eq:press-char}
	Q_h= \divergence X_{h,\sim} = \divergence V_h = \mathcal{L}^0_{k-1}(\mathcal{T}_h) \cap L^2_0(\Omega).
\end{align}
All experiments are conducted with polynomial degree $k = 4$, i.e., we have $X_h = \mathcal{L}^1_{4}(\mathcal{T}_h)^2$, see~\eqref{def:Xh-SV}, and $Q_h = \mathcal{L}^0_{3}(\mathcal{T}_h)\cap L^2_0(\Omega)$. 

For the discrete problem~\eqref{eq:Stokes-mixed-inh} we impose that $\tr(u_h)=\tr(g_h)$ using the FEniCS routine \verb|DirichletBC|, instead of reducing the problem to the corresponding homogeneous one.
 We investigate the cases 
\begin{center}
 (a) \;\; $g_h = \tilde I_h g$ \qquad and  \qquad\qquad (b) \;\; $g_h = L_h g$,
\end{center}
with $\tilde I_h$ as in~\eqref{def:tilde-Ih}, and $L_h$ the Lagrange interpolation. 
Then the mixed problem for the Scott--Vogelius element is solved in FEniCS. 

\subsubsection{Convergence and compatibility}

We investigate the error of the exact solution \eqref{eq:exact_sol} and the discrete solution to the inhomogeneous Stokes problem with the Scott--Vogelius element. 
Specifically, the effect of the choice of $g_h$, as $\tilde I_h g$ and of $L_h g$ is examined, as shown in Tables~\ref{tbl:SV_direct_withbdcor} and~\ref{tbl:SV_direct_nobdcor}, respectively. 
Recall that the first choice ensures compatibility~\eqref{eq:comp-gh} with the exact divergence constraint $\divergence u_h = 0$, whereas the latter does not.  
This is reflected by the fact that the divergence constraint is not satisfied exactly for $g_h = L_h g$, see Table~\ref{tbl:SV_direct_nobdcor}, which is particularly visible for coarser meshes.
For example for parameter $N=16$ we have $\norm{\divergence u_h}_{L^2(\Omega)}=$3.98e-05 for $g_h = L_h g$, see Table~\ref{tbl:SV_direct_nobdcor}, in comparison to $\norm{\divergence u_h}_{L^2(\Omega)}=$6.60e-11 when the boundary interpolation is corrected by using $g_h = \tilde I_h g$, see Table~\ref{tbl:SV_direct_withbdcor}.
As expected, as the mesh parameter decreases, the divergence norm of $u_h$ reduces. 
Note that not using the corrected boundary interpolation \eqref{def:tilde-Ih} does not necessarily imply a high value in the $L^2$-norm of the discrete divergence, but there is no guarantee for it to be small. The $L^2$-error of the divergence is trivially bounded above by the $H^1$-norm of the velocity error and hence reduces as $N$ increases, for sufficiently large $N$. 

\begin{table}[h]
	\begin{tabular}{rccccc}
		\toprule
		\;N & $h$\; & \;$\norm{u - u_h}_{L^2(\Omega)}$ \;& \;$\norm{u - u_h}_{H^1(\Omega)}$ \;& \;$\norm{\divergence u_h}_{L^2(\Omega)}$\; & \;$\norm{p - p_h}_{L^2(\Omega)}$ \;\\
		\midrule
		8 & 1.58e-01 &    1.30e-01 &    1.22e+01  & 2.89e-11   &      3.59e+01 \\
		16 & 8.24e-02  &   2.95e-03  &   5.78e-01  & 6.60e-11   &      1.72e+00\\
		32 & 3.98e-02  &   1.66e-04   &  5.67e-02 &  4.71e-08   &      2.00e-01\\
		64 & 2.06e-02   &  3.80e-06  &  2.85e-03  & 4.38e-09      &   9.13e-03\\
		128 & 9.94e-03  &   9.52e-08  &   1.52e-04 &  1.26e-09    &     3.95e-04\\
		\bottomrule
	\end{tabular}	
	\caption{
		Error norms for 	Scott--Vogelius discrete solutions $(u_h, p_h)$ to~\eqref{eq:Stokes-mixed-inh} with $k = 4$ and $g_h = \tilde I_h g$, and manufactured solution $(u,p)$~\eqref{eq:exact_sol} to the Stokes problem~\eqref{eq:Stokes-weak-inh}, with $\Ra = 1$ on the modified mesh $\tilde{\tria}_h$, cf. Section~\ref{sec:mesh_mod}. 
	}		
	\label{tbl:SV_direct_withbdcor}
\end{table}

\begin{table}[h]
	\begin{tabular}{rccccc}
		\toprule
		\;N & $h$ \;& \;$\norm{u - u_h}_{L^2(\Omega)}$ \;& \;$\norm{u - u_h}_{H^1(\Omega)}$ \;& \;$\norm{\divergence u_h}_{L^2(\Omega)}$\; & \;$\norm{p - p_h}_{L^2(\Omega)}$ \;\\
		\midrule
		 8 & 1.58e-01  &   1.30e-01   &  1.22e+01 &  2.88e-04  &       3.59e+01 \\
		 16 &  8.24e-02  &   2.95e-03  &   5.78e-01  & 3.98e-05   &      1.72e+00 \\
		 32 & 3.98e-02   &  1.70e-04   &  5.67e-02 &  1.55e-03    &     2.00e-01 \\
		 64 & 2.06e-02  &   3.80e-06   &  2.85e-03 &  1.98e-06    &     9.13e-03 \\
		 128 & 9.94e-03 &    9.52e-08  &   1.52e-04 &  4.04e-08    &    3.95e-04 \\
		\bottomrule
	\end{tabular}
	\caption{Error norms for 	Scott--Vogelius discrete solutions $(u_h, p_h)$ to~\eqref{eq:Stokes-mixed-inh} with $k = 4$ and $g_h = L_h g$, and manufactured solution $(u,p)$~\eqref{eq:exact_sol} to the Stokes problem~\eqref{eq:Stokes-weak-inh}, with $\Ra = 1$ on the modified mesh $\tilde{\tria}_h$, cf. Section~\ref{sec:mesh_mod}. 
	}
	\label{tbl:SV_direct_nobdcor}
\end{table}

\subsubsection{Pressure robustness}\label{sec:press-rob}

Here, we verify the pressure robustness property of the Scott--Vogelius element for the Stokes problem with inhomogeneous Dirichlet boundary conditions when employing the compatible datum $g_h = \tilde I_h g$. 
We compare it to the Taylor--Hood pair of finite element spaces consisting of $X_h = \mathcal{L}^1_4(\tria_h)^2$ and  $Q_h=\mathcal{L}^1_{3}(\tria_h)$, which is known to be not pressure-robust. 

We use the manufactured solution $(u,p)$ to the Stokes problem~\eqref{eq:Stokes-weak-inh} given in~\eqref{eq:exact_sol}. 
This extends Example 1.1 in~\cite{John2017} from homogeneous to inhomogeneous boundary conditions. We consider the parameters~$\Ra \in \{10^{11}, 10^{12}, 10^{13}\}$ scaling the given function $F$. 
For the exact solutions the parameter $\Ra$ only affects the pressure $p$ and not the velocity, see~\eqref{eq:exact_sol}.

Pressure robustness is achieved by the discrete solution computed by the Scott--Vogelius method: the velocity error remains almost the same for all values of $\Ra$, while the pressure error scales with $\Ra$ as expected. 
In contrast, the Taylor--Hood finite element pair yields a higher discrete velocity error that scales with the pressure error, see Figure~\ref{fig:SV_TH_direct_pressure_robustness_loglog}. 

\input{fig/H1u_L2p_errorplot_Ras_SV_TH_direct_loglog.tex}

\section{Iterated penalty method}\label{sec:ipm}

As mentioned before, the Scott--Vogelius element yields exactly divergence-free velocity approximations. This is achieved by choosing the pressure space as the divergence of the velocity space. 
The difficulty, however, lies in the characterisation and implementation of the pressure space, which depends on the geometry of the mesh. 
Indeed, the pressure space may be smaller than $\mathcal{L}^0_{k-1}(\mathcal{T}_h)$ due to the presence of singular vertices, see~\cite{GuzmanScott2019}. 

A customary remedy consists in using an Uzawa-type penalisation method based on an iteration. 
The velocity and the pressure approximations are updated in turn and the divergence is penalised. 
Such methods can be applied to general inf-sup stable mixed finite element spaces, see~\cite[Ch.~7.2]{BernardiGiraultRaviartHechtRiviere2024} for a version for inhomogeneous Dirichlet boundary conditions. 

\subsubsection*{Iteration for general mixed finite element pair}
Before giving an outline of the remaining section, let us introduce the Uzawa-type methods.
Let $(X_h, Q_h)$ be a pair of finite element spaces as defined in~\eqref{def:Xh}, \eqref{def:Qh}, let $V_h = X_h\cap H_0^1(\Omega)^d$ as in~\eqref{def:Vh}. Furthermore, let $F \in H^{-1}(\Omega)^d$ and let $g_h \in X_{h} $ be a suitable approximation of $g \in \Hsim$.  
For the Uzawa method, to approximate the solution $(u_h,p_h) \in X_h \times Q_h$ to the Stokes problem we choose parameters $\rho_u\geq 0$, $\rho_p>0$ and set $p_{h,0} \coloneqq 0$.
In the $i$th iteration, for $i \in \mathbb{N}$, and given $p_{h,i-1}\in Q_h$ we want to find $u_{h,i} \in X_h$ with $\tr(u_{h,i} ) = \tr(g_h)$ such that 
\begin{subequations}\label{eq:Uzawa}
	\begin{align}\label{eq:Uzawa-u}
		\nu \skp{\nabla u_{h,i}}{\nabla v_h} + \rho_u \skp{\divergence u_{h,i}}{\divergence v_h}  & = \skp{F}{v_h} + \skp{p_{h,i-1}}{\divergence v_h} \quad &&\text{for all } v_h \in V_h,
		\end{align}
		and then we update $p_{h,i} \in Q_h$ by 
		\begin{align}
		\label{eq:Uzawa-p}
		\qquad\qquad \skp{p_{h,i}}{q_h} &=  \skp{p_{h,i-1}}{q_h} - \rho_p \skp{ \divergence u_{h,i}}{q_h}  \quad &&\text{for all } q_h \in Q_h.
	\end{align}
\end{subequations}
Using the Uzawa method, in each iteration a smaller system is solved compared to the mixed method~\eqref{eq:Stokes-mixed-inh}. Thanks to coercivity of the bilinear form representing the left-hand side in \eqref{eq:Uzawa-u}, well-posedness is available for any $\rho_u \geq 0$.

The special case of $\rho_u = 0$ leads to an Uzawa method for the standard Lagrangian, whereas $\rho_u>0$ corresponds to the case of an augmented Lagrangian, where the divergence is penalised. Note that the first velocity iterate $u_{h,1}$ corresponds to the solution to the Stokes equation with grad-div stabilisation~\cite{Franca1988}.
In the latter case, for any $0 < \rho_p< 2 \rho_u$ convergence of $\nabla u_{h,i} \to \nabla u_h$ and of $p_{h,i} \to p_h$ in $L^2(\Omega)$, as $i \to \infty$,  holds for $(u_h,p_h) \in X_h \times Q_h$ the solutions to~\eqref{eq:Stokes-mixed-inh}, see for example~\cite[Thm.~7.2.5]{BernardiGiraultRaviartHechtRiviere2024}. 
A common choice is $\rho_p = \rho_u > 0$, see, e.g.,~\cite[Sec.~3.2.5]{Thomasset1981} and~\cite{Ainsworth2023}, and in the following, we shall adopt this choice. 
Let us note, however, that better rates of convergence can be obtained by choosing parameter $\rho_p^i$ depending on the solution of the current iteration,
but this increases the computational cost, see~\cite[Ch.~2.3]{Fortin1983}, \cite[Alg.~7.2]{BernardiGiraultRaviartHechtRiviere2024}. 
For constant values (in the iteration), the choice  $\rho_p= \nu + \rho_u$ is optimal regarding convergence rate in the pressure error norm, see~\cite{Nochetto2004}. 
In this case, the contraction constant for the pressures is $(1 - \overline{\beta}^2)^{1/2}$ with $\overline{\beta}$ the inf-sup constant in~\eqref{def:discr-inf-sup-cond}, see also~\cite[ Remark 7.2.6]{BernardiGiraultRaviartHechtRiviere2024}. 
Note that, in general, \eqref{eq:Uzawa-p} cannot be used to eliminate $p_{h,i-1}$
in~\eqref{eq:Uzawa-u}. 
Hence, the general convergence results are based on a contraction in the pressure rather than in the velocity.   

In the following, we choose $\rho_u= \rho_p = \rho>0$. 
Since all of this holds for general inf-sup stable finite element pairs, it applies  in particular to the Scott--Vogelius element under the condition of inf-sup stability. 
\medskip 

For the special case of the Scott--Vogelius element the Uzawa method~\eqref{eq:Uzawa} is also referred to as \emph{iterated penalty method} (IPM), see~\cite[Sec.~13.1]{Brenner2008}. 
The advantage over the mixed formulation for the Scott--Vogelius element in~\eqref{eq:Stokes-mixed-inh} is that the pressure system \eqref{eq:Uzawa-p} reduces to an update and that one can work directly with divergence of velocity functions. Hence, no pressure basis is needed. 

As in Section~\ref{sec:SV} we consider for the velocity the spaces  
\begin{align*}
X_h = \mathcal{L}^1_{k}(\mathcal{T}_h)^d, \quad 	X_{h,\sim} = X_h\cap \Hsim,\quad V_h = X_h \cap H_0^1(\Omega)^d, 
\end{align*}
and thanks to Assumption~\ref{assumpt:pressure-sp} we have 
$$Q_h= \divergence V_ h = \divergence X_{h,\sim}.
$$
In this case~\eqref{eq:Uzawa-p} reduces to 
\begin{align}\label{eq:IPM-hom-p}
	p_{h,i} \coloneqq p_{h,i-1} - \rho_p  \divergence u_{h,i}.
\end{align}
As before we choose the discrete datum $g_h \in X_{h,\sim}$, i.e., it is compatible with the exact divergence constraint. 
Then, the Uzawa method in~\eqref{eq:Uzawa} reduces to the following. 

\subsubsection*{IPM for the Scott--Vogelius element} 
We consider a parameter $\rho>0$. Starting from $\oldphi_{h,0}=0$ in the $i$th iteration for given $\oldphi_{h,i-1}\in X_h$ we want to find 
$u_{h,i}\in X_{h,\sim}$ with $u_{h,i} = g_h$ on $\partial\Omega$ such that
\begin{subequations}\label{eq:ipm}
	\begin{equation}\label{eq:ipm-u}
		\begin{aligned}
			\nu \skp{\nabla u_{h,i}}{\nabla v_h} + \rho \skp{\divergence u_{h,i}}{\divergence v_h}  
			& = \skp{F}{v_h} + \skp{\divergence \oldphi_{h,i-1}}{\divergence v_h} \qquad \text{for all } v_h \in V_h, 
		\end{aligned}
	\end{equation}
	and update 
\begin{equation}\label{eq:ipm-phi}
	\oldphi_{h,i} \coloneqq \oldphi_{h,i-1} - \rho u_{h,i}.
\end{equation}	
\end{subequations}
Note that the iteration only involves velocity functions, with $u_{h,i}, \oldphi_{h,i} \in X_{h,\sim}$, and the pressure iterates can be determined by 
\begin{align}\label{eq:ipm-p}
	p_{h,i}\coloneqq \divergence \oldphi_{h,i}\in \divergence X_{h,\sim} = Q_h. 
\end{align}
For a given tolerance $\tol >0$ we use the stopping criterion
\begin{align}\label{eq:tol}
	\norm{\divergence u_{h,i} }_{L^2(\Omega)} < \tol. 
	\end{align}
Equivalently to~\eqref{eq:ipm}, one may solve an inhomogeneous problem in the first iteration, and then solve homogeneous problems in all subsequent steps by solving for $u_{h,i}-u_{h,i-1} \in V_h$ instead, see~\cite[Alg.~7.2]{BernardiGiraultRaviartHechtRiviere2024}. 
\subsubsection*{Context} 
Most results on the IPM~\eqref{eq:ipm} are treating the case of homogeneous Dirichlet boundary conditions, and convergence results are presented in~\cite{Brenner2008}, or for mixed boundary conditions with homogeneous Dirichlet boundary conditions on a part of the boundary and natural boundary conditions on the rest, see~\cite{Ainsworth2023}. 
The only exception we are aware of is~\cite{BernardiGiraultRaviartHechtRiviere2024} for an Uzawa algorithm~\eqref{eq:Uzawa} for general mixed finite element methods. 

For the case of homogeneous boundary conditions, the Uzawa algorithm~\eqref{eq:Uzawa} and the IPM~\eqref{eq:ipm} have been extended also to the incompressible Navier-Stokes equations, see~\cite{Chen2015} and~\cite{Geredeli2023}, respectively. 

\begin{remark}
	\begin{enumerate}[label = (\alph*)]
		\item (conditions on $g_h$)
		Recall that  the mixed method for the Scott--Vogelius element is well-posed for boundary data $g_h\in X_h$. 
		Hence, if exact divergence constraints are not of particular interest, the compatibility $g_h \in X_{h,\sim}$  is not needed.  
		But to ensure exact divergence constraints, we require $g_h \in X_{h,\sim}$, see Remark~\ref{rmk:ex-div}. 
		Since for the IPM (for the Scott--Vogelius element) the divergence is penalised and serves as stopping criterion, it is essential to choose the boundary data $g_h \in X_{h,\sim}$ in a compatible manner. 
	\item (constructive divergence-free extension)
		The IPM can be seen as a way to construct a divergence-free extension (sometimes also referred to as \emph{lifting operator}) of $\tr(g_h)$.  	
	\end{enumerate}
\end{remark}
In the following, we investigate the IPM for inhomogeneous boundary conditions. 
In Section~\ref{sec:IPM-conv}, we extend the convergence results in~\cite{Brenner2008} to inhomogeneous Dirichlet boundary conditions. This is not much more demanding than in the homogeneous case, but shall be presented for the sake of completeness. 
In combination with the quasi-optimality results obtained in Section~\ref{sec:SV} this allows us to derive asymptotically pressure-robust error estimates for the IPM solution in Section~\ref{sec:ipm-error}. Since the error estimates involve only best approximation errors they do not require higher regularity on the solutions. Due to the same norms appearing, they are quasi-optimal in $(u,p)$, and the lack of pressure robustness is quantified.  
Furthermore, the conformity with respect to the divergence allows us to show also monotonicity of the divergence norm in Section~\ref{sec:mon-div} employing energy arguments. 
Finally, in Section~\ref{sec:ipm-num-exp} we present numerical experiments confirming our theoretical results.  
\subsection{Convergence}
\label{sec:IPM-conv}

Since we have $\divergence X_{h,\sim} =Q_h$ for the Scott--Vogelius element, for the IPM convergence results based on a velocity contraction are possible, which we state for the readers convenience.
The extension from homogeneous to inhomogeneous Dirichlet boundary conditions follows from~\cite[Ch. 13.1]{Brenner2008} with minor modifications. 

Before stating the convergence result let us note that with $\overline \beta$  the inf-sup constant as in~\eqref{def:discr-inf-sup-cond} it has been shown in~\cite[Lem.~(12.5.10)]{Brenner2008}, see also~\cite[Sec. 3.5, eq.~(3.163)]{Elman2014}, that 
\begin{align}\label{eq:beta}
\overline \beta \norm{\nabla v_h}_{L^2(\Omega)} \leq  \norm{\divergence v_h}_{L^2(\Omega)}  
\quad \text{ for any } v_h \in V_{h,\divergence}^\perp,
\end{align}
with $V_{h,\divergence}^\perp \coloneqq \{v_h \in V_{h} \colon \skp{\nabla v_h}{\nabla w_h} = 0 \; \forall w_h \in V_{h,\divergence} \}$. 
In fact, this is a direct consequence of the stability estimate~\eqref{eq:Bog-h-stab} in combination with the fact that $Q_h = \divergence V_h$: Indeed, for $\xi_h \in V_{h,\divergence}^\bot$ we have $\divergence \xi_h \in Q_h$, and thus, by~\eqref{eq:bog-mixed} it follows that $\mathcal{B}_h(\divergence \xi_h) = \xi_h$. 
In combination with the fact $\skp{\divergence \mathcal{B}_h(\divergence \xi_h)}{q_h} =\skp{\divergence \xi_h}{q_h} $  for any $q_h \in Q_h$, and thus, $\divergence \mathcal{B}_h(\divergence \xi_h) =\divergence \xi_h$ the claim follows from~\eqref{eq:Bog-h-stab}. 
{\begin{lemma}[convergence of the IPM]\label{lem:conv_ipm}
Let~$\nu >0$, let $F \in H^{-1}(\Omega)^d$ be given. 
For some $k \in \mathbb{N}$ let $X_h, X_{h,\sim,\divergence},V_h, Q_h$ be the discrete spaces for the Scott--Vogelius element as in~\eqref{def:Xh-SV}, \eqref{def:Xhdiv-SV},~\eqref{def:Vh_SV} and~\eqref{def:Qh_SV}, let Assumption~\ref{assumpt:FEM} hold with inf-sup constant $\overline \beta>0$ and let Assumption~\ref{assumpt:pressure-sp} be satisfied.
For some $g_h \in X_{h,\sim}$ let $(u_h,p_h) \in X_{h,\sim,\divergence} \times Q_h \subset \Hdiv \times \Lmeanfree(\Omega)$ be the discrete solution to the discrete mixed problem~\eqref{eq:Stokes-mixed-inh}. 
Furthermore, for $\rho>0$ and $\oldphi_{h,0}=0$ let $(u_{h,i})_{i \in \mathbb{N}}, (\oldphi_{h,i})_{i \in \mathbb{N}} \subset  X_{h,\sim}$ be sequences of iterate solutions to the IPM~\eqref{eq:ipm}. 
Then, with 
$$\theta\coloneqq \frac{\nu}{\nu+\rho\overline \beta^2 }<1$$   
for any $i \in \mathbb{N}$ one has that
\begin{align} \label{est:IPM-conv-grad}
		\norm{\nabla u_h - \nabla u_{h,i}}_{L^2(\Omega)} 
		&\leq \theta^{i-1}~ \norm{\nabla u_h - \nabla u_{h,1}}_{L^2(\Omega)},
		 \\
		\label{est:IPM-conv-pi}
		 \norm{p_h - \divergence \oldphi_{h,i} }_{L^2(\Omega)} &\leq 
		 \tfrac{\nu + \rho \overline \beta \sqrt{d}}{\overline \beta} ~ \theta^{i-1} \,\norm{\nabla u_h - \nabla u_{h,1}}_{L^2(\Omega)}.
	\end{align}
\end{lemma}}
\begin{proof}
	
	\textit{1.~Step (contraction for velocity):}
Applying~\eqref{eq:Stokes-mixed-inh} and~\eqref{eq:ipm} the error  $e_{h,i}\coloneqq u_h - u_{h,i}$ satisfies 
	\begin{align}\label{eq:conv-1}
		\nu \skp{\nabla e_{h,i}}{\nabla v_h} + \skp{\divergence \oldphi_{h,i-1}-p_h}{\divergence v_h} - \rho \skp{\divergence u_{h,i}}{\divergence v_h} = 0 \quad \text{ for any } v_h\in V_h. 
	\end{align}
Since by~\eqref{eq:Stokes-mixed-inh} the function $u_h$ is discretely divergence-free, we obtain
	\begin{align}\label{eq:conv-2}
		\nu \skp{\nabla e_{h,i}}{\nabla v_h} + \skp{\divergence \oldphi_{h,i-1}-p_h}{\divergence v_h} + \rho \skp{\divergence e_{h,i}}{\divergence v_h} =0 \quad \text{ for any } v_h\in V_h. 
	\end{align} 
Now the proof proceeds as in the homogeneous case in \cite[Ch. 13.1]{Brenner2008}. 
Employing the update~\eqref{eq:ipm-phi}  $$\oldphi_{h,i-1}=\oldphi_{h,i-2}-\rho u_{h,i-1}$$ in~\eqref{eq:conv-2} and using~\eqref{eq:conv-1} shows that we have  for any $v_h\in V_h$ 
	\begin{align*}
		\nu \skp{\nabla e_{h,i}}{\nabla v_h}  + \rho \skp{\divergence e_{h,i}}{\divergence v_h} &= -\skp{\divergence \oldphi_{h,i-2}-p_h}{\divergence v_h} + \rho\skp{\divergence u_{h,i-1} }{\divergence v_h} \\
		&= \nu \skp{\nabla e_{h,i-1}}{\nabla v_h}.
	\end{align*}
	Since $\tr(u_h)=\tr(u_{h,i}) = \tr(g_h)$ we may choose  $v_h= e_{h,i} \in V_h$ as test function.  
	Noting that by~\eqref{eq:conv-1} we have that $e_{h,i} \in V_{h,\divergence}^\bot$, and thus, with~\eqref{eq:beta} we obtain
	\begin{equation}\label{eq:proof_conv_IPM}
	\begin{aligned}
		(\nu + \rho\overline \beta^2) \norm{\nabla e_{h,i}}_{L^2(\Omega)}^2  
		&\leq 	\nu \norm{\nabla  e_{h,i}}_{L^2(\Omega)}^2   + \rho  \norm{\divergence e_{h,i}}_{L^2(\Omega)}^2  \\
		&= \nu \skp{\nabla e_{h,i-1}}{\nabla e_{h,i}}
		 \leq \nu \norm{\nabla e_{h,i-1}}_{L^2(\Omega)}  \norm{\nabla e_{h,i}}_{L^2(\Omega)}.  
	\end{aligned}
	\end{equation}
	If $\norm{\nabla e_{h,i}}_{L^2(\Omega)} = 0$, then the claim holds trivially, and otherwise we have with $\theta =\frac{\nu}{\nu+\rho\overline \beta^2 }< 1 $ that 
	\begin{align}\label{eq:conv-3}
		\norm{\nabla e_{h,i}}_{L^2(\Omega)} \leq \theta \norm{\nabla e_{h,i-1}}_{L^2(\Omega)} \leq \theta^{i-1} \norm{\nabla e_{h,1}}_{L^2(\Omega)},
	\end{align}
	which shows the first claim. 
	
		\textit{2.~Step (pressure estimate):}
By~\eqref{eq:conv-2} for $i \geq 1$ we have  the relation 
 \begin{align}\label{eq:proof_conv_IPM_pi}
	\skp{\divergence \oldphi_{h,i-1}-p_h}{\divergence v_h} = -\nu \skp{\nabla e_{h,i}}{\nabla v_h} - \rho \skp{\divergence e_{h,i}}{\divergence v_h} \quad\text{for all } v_h\in V_h.
\end{align}
Recall from Assumption~\ref{assumpt:pressure-sp} that we have $\divergence X_{h,
	\sim} = \divergence V_h = Q_h$, hence we know that $\divergence \oldphi_{h,i-1} - p_h \in Q_h \subset L^2_0(\Omega)$. 
Hence, with the discrete \Bogovskii{} operator $\mathcal{B}_h\colon L^2_0(\Omega) \to V_h$ in Remark~\ref{rmk:FortinBogovskii}~\ref{itm:discrBog} the function $v_h =  \mathcal{B}_h(\divergence \oldphi_{h,i-1}-p_h)$ is an admissible test function in~\eqref{eq:proof_conv_IPM_pi}. 
Additionally using also the divergence preservation in the dual of $Q_h$ as in~\eqref{eq:Bog-h} with $\divergence \oldphi_{h,i} - p_h, \divergence e_{h,i} \in \divergence X_{h,\sim}  = Q_h$ yields 
\begin{equation}
\begin{aligned}\label{eq:conv-4}
	 \norm{\divergence \oldphi_{h,i-1}-p_h}_{L^2(\Omega)}^2   &= 	\skp{\divergence \oldphi_{h,i-1}-p_h}{\divergence \oldphi_{h,i-1}-p_h}\\
&	=
	-\nu \skp{\nabla e_{h,i}}{\nabla \mathcal{B}_h(\divergence \oldphi_{h,i-1}-p_h)}  
	- \rho \skp{\divergence e_{h,i}}{\divergence \oldphi_{h,i-1}-p_h}. 
\end{aligned}
\end{equation}
Then, applying the stability of $\mathcal{B}_h$ in~\eqref{eq:Bog-h-stab} it follows that 
	\begin{alignat*}{3}
	\norm{\divergence \oldphi_{h,i-1}-p_h}_{L^2(\Omega)}^2 
	&\leq &&~\nu  \norm{\nabla e_{h,i}}_{L^2(\Omega)} \norm{\nabla \mathcal{B}_h(\divergence \oldphi_{h,i-1}-p_h)}_{L^2(\Omega)} \\
	&&& + \rho \norm{\divergence e_{h,i}}_{L^2(\Omega)} \norm{\divergence \oldphi_{h,i-1}-p_h}_{L^2(\Omega)}\\
	&\leq &&~\tfrac{\nu}{\overline \beta}  \norm{\nabla e_{h,i}}_{L^2(\Omega)} \norm{\divergence \oldphi_{h,i-1}-p_h}_{L^2(\Omega)} \\
	&&& + \rho \norm{\divergence e_{h,i}}_{L^2(\Omega)} \norm{\divergence \oldphi_{h,i-1}-p_h}_{L^2(\Omega)}. 
\end{alignat*}
In the non-trivial case, applying~\eqref{eq:beta} and the estimate
\begin{align*}
	\norm{\divergence e_{h,i}}_{L^2(\Omega)} \leq \sqrt{d} \norm{\nabla e_{h,i}}_{L^2(\Omega)}
\end{align*}
yields with \eqref{eq:conv-3} that
\begin{equation}
\begin{aligned}
	\norm{\divergence \oldphi_{h,i-1}-p_h}_{L^2(\Omega)}
	& \leq
	\tfrac{\nu}{\overline \beta}  \norm{\nabla e_{h,i}}_{L^2(\Omega)}  + \rho \norm{\divergence e_{h,i}}_{L^2(\Omega)}\\
	& \leq \tfrac{\nu + \rho \overline \beta \sqrt{d}}{\overline \beta} \norm{\nabla e_{h,i}}_{L^2(\Omega)}
	\leq \tfrac{\nu + \rho \overline \beta \sqrt{d}}{\overline \beta}
	 \theta^{i-1}
	\norm{\nabla e_{h,1}}_{L^2(\Omega)}. 
\end{aligned}
\end{equation}
This shows the claim.  
\end{proof}
Besides showing convergence, Lemma~\ref{lem:conv_ipm} ensures that the iteration terminates when using the stopping criterion~\eqref{eq:tol}.

\begin{remark}
Note that the best possible contraction constant obtained in~\cite{Nochetto2004} and \cite[Rmk. 7.2.6]{BernardiGiraultRaviartHechtRiviere2024} is $1 - \overline \beta^2$. 
	This means that the rate obtained from the pressure based estimates cannot be better than $1 - \overline \beta^2$, which is close to one for rather small values of $\overline \beta$.
On the other hand, the contraction constant in Lemma~\ref{lem:conv_ipm} is $\frac{\nu}{\nu + \rho \overline \beta^2}$, which reduces to $\frac{1}{1 + \gamma \overline \beta^2}$ for $\rho = \gamma \nu$ for some $\gamma>0$. 
This factor can be arbitrarily small for  large $\rho$. Clearly, large values of $\rho$ are not recommended because it also affects the condition number. 
Even for moderate $\rho$ the contraction constant can be expected to be sharper, if $\overline \beta \ll 1$.  
Recall that for working in the velocity based formulation the property $\divergence X_{h,\sim} = Q_h$ is key. 
\end{remark}

\subsection{Asymptotically pressure-robust error estimates}
\label{sec:ipm-error}

In this section in Theorem~\ref{thm:ipm-quasi-opt} we present an error estimate for the IPM in terms of best approximation errors. 
It holds under minimal regularity assumptions and takes both the space discretisation and the iteration into account. 
It is quasi-optimal in the sense that the same norms occur on both sides if the velocity and pressure errors are considered together. 
Furthermore, it is asymptotically pressure-robust in the limit of high iteration numbers in the IPM. 
The effect of the iteration number on the lack of pressure robustness is quantified. 

This can be achieved by splitting the error between exact solution $u \in H^1(\Omega)^d$ to~\eqref{eq:Stokes-weak-inh} and IPM approximate solution $u_{h,i}\in X_h$ as in~\eqref{eq:ipm} into 
\begin{align}\label{eq:error-split}
	u-u_{h,i} &= u- u_h + u_h - u_{h,i},
\end{align}
where $u_h \in X_h$ is the solution to the mixed problem~\eqref{eq:Stokes-mixed-inh} via the Scott--Vogelius method. 
We use the quasi-optimality result for the mixed Scott--Vogelius method in Theorem~\ref{thm:SV-inhom} on $u - u_h$ and the convergence result for the IPM in Lemma~\ref{lem:conv_ipm} above to estimate $u_h - u_{h,i}$.
Then, it remains to further estimate the error of the first IPM iterate $u_h - u_{h,1}$. 

This corresponds to the estimate for the grad-div stabilisation; see e.g.~\cite[Thm.~1]{JenkinsJohnLinkeEtAl2014}. 

\begin{lemma}{(error estimate for first IPM iterate)}\label{lemma:quasi-opt_grad-div}
Let the assumptions of Theorem~\ref{thm:SV-inhom} be satisfied, i.e., in particular let $g_h \in X_{h,\sim}$ satisfy $\tr(g_h)=\tr(\Pi_h(g))$ for $\Pi_h$ as in Proposition \ref{prop:SV-op}. Furthermore, let the assumptions in  Lemma~\ref{lem:conv_ipm} hold and let $u_{h,1} \in X_{h,\sim}$ be the velocity function obtained in the first iteration of the IPM in~\eqref{eq:ipm}. 

Then, for the velocity error one has
 \begin{align*}
 	\norm{\nabla u - \nabla u_{h,1}}_{L^2(\Omega)} \leq c_1 \inf_{\substack{v_h\in X_h}} \norm{\nabla u -\nabla v_h}_{L^2(\Omega)} + c_2 \norm{p}_{L^2(\Omega)} ,
 \end{align*}
  with $c_1 = 2(1 + c_F)$, $c_2 = \frac{1}{\sqrt{2 \rho \nu}} $ and  $c_F>0$ is the constant in Assumption~\ref{assumpt:Fortin}~\ref{itm:stab}.
\end{lemma}
\begin{proof} The proof proceeds similarly as the one of quasi-optimality of the Stokes problem see Theorem~\ref{thm:quasiopt-mixed} and also~\cite[Sec.~3]{John2017}. 
Additionally we have to account for the inhomogeneous boundary conditions. 

The first IPM iterate $u_{h,1} \in X_h$ with $\tr(u_{h,1}) = \tr(g_h)$ solves~\eqref{eq:ipm-u} with $\oldphi_{h,0} = 0 $, and the exact solution $u \in H^1_{0,\divergence}(\Omega)$ such that $\tr(u) = \tr(g)$ solves~\eqref{eq:Stokes-weak-inh}. 
Therefore, the error $e_{h,1} \coloneqq u - u_{h,1}$ satisfies 
\begin{align}\label{eq:ipm-first-1}
			\nu\skp{\nabla e_{h,1}}{\nabla v_h} - \skp{p - \rho \divergence e_{h,1}}{\divergence v_h} = 0 \qquad \text{for all } v_h\in V_h \subset H^1_0(\Omega)^d.
\end{align} 	
We split the error $e_{h,1}$ into a discrete part with zero trace and into the rest. I.e., for arbitrary $z_h \in X_h$ with $\tr(z_h) = \tr(g_h)$ we consider 
	\begin{equation}\label{eq:ipm-first-2}
	e_{h,1} = u - u_{h,1} = u-z_h - ( u_{h,1} - z_h) \eqqcolon w - v_h,
\end{equation}
and we note that $v_h = u_{h,1} - z_h \in V_h$. 
Thus, using~\eqref{eq:ipm-first-1} we find  that
\begin{align}\label{eq:ipm-first-3}
	\nu \norm{\nabla v_h}_{L^2(\Omega)}^2 
	=
	\nu  \skp{\nabla w- \nabla e_{h,1}}{\nabla v_h}
	\leq 
	\nu \norm{\nabla w}_{L^2(\Omega)} \norm{\nabla v_h}_{L^2(\Omega)} - \skp{p - \rho \divergence e_{h,1}}{\divergence v_h}. 
\end{align}
Since $\divergence u= 0$ we also have with \eqref{eq:ipm-first-2} that
\begin{align}\label{eq:ipm-first-4}
	\rho\norm{\divergence v_h}_{L^2(\Omega)}^2 = \rho\skp{\divergence w - \divergence e_{h,1}}{\divergence v_h} 
	= - \rho \skp{\divergence (z_h + e_{h,1})}{\divergence v_h}.
\end{align}
Adding~\eqref{eq:ipm-first-3} and~\eqref{eq:ipm-first-4} and applying Young's inequality we arrive at 
\begin{alignat*}{3}	
\nu \norm{\nabla v_h}_{L^2(\Omega)}^2 + 
		\rho\norm{\divergence v_h}_{L^2(\Omega)}^2
&\leq 
&&\nu \norm{\nabla w}_{L^2(\Omega)} \norm{\nabla v_h}_{L^2(\Omega)} - \skp{p + \rho \divergence z_{h}}{\divergence v_h}\\
& \leq 
&& \tfrac{\nu}{2} \norm{\nabla w}_{L^2(\Omega)}^2 + \tfrac{\nu}{2} \norm{\nabla v_h}_{L^2(\Omega)}^2 \\
&&& + \tfrac{1}{4\rho} \norm{p + \rho \divergence z_h}_{L^2(\Omega)}^2 + \rho\norm{\divergence v_h}_{L^2(\Omega)}^2.
\end{alignat*}
Hence, by rearranging the terms it follows with $w=u-z_h$ that 
\begin{align}\label{eq:ipm-first-5}
	\norm{\nabla v_h}_{L^2(\Omega)} 
	\leq \norm{\nabla u - \nabla z_h}_{L^2(\Omega)} +
	\frac{1}{\sqrt{2\rho\nu}}  \norm{p+\rho\divergence z_h}_{L^2(\Omega)}. 
\end{align}
Then, by triangle inequality and~\eqref{eq:ipm-first-2} and applying~\eqref{eq:ipm-first-5} we arrive at 
\begin{align*}
	\norm{\nabla e_{h,1}}_{L^2(\Omega)}  
	&\leq 
	\norm{\nabla u - \nabla z_h}_{L^2(\Omega)} + \norm{\nabla v_h}_{L^2(\Omega)} \\
	&\leq 2 \norm{\nabla u -\nabla z_h}_{L^2(\Omega)} + \frac{1}{\sqrt{2\rho\nu}} \norm{p + \rho\divergence z_h }_{L^2(\Omega)}. 
\end{align*}
Since $z_h \in X_h$ with $\tr(z_h)=\tr(g_h)$ was arbitrary, this implies  
\begin{align}\label{eq:ipm-first-6}
	\norm{\nabla e_{h,1}}_{L^2(\Omega)}  \leq \inf_{\substack{z_h\in X_h \\ \tr(z_h)=\tr(g_h)}} \left(2 \norm{\nabla u -\nabla z_h}_{L^2(\Omega)} + \frac{1}{\sqrt{2\rho\nu}} \norm{p + \rho\divergence z_h}_{L^2(\Omega)} \right).
\end{align}
Note that we have $\tr(g_h) = \tr(\Pi_h(g))$ with Fortin operator $\Pi_h \colon H^1(\Omega)^d\to X_h$ as in Assumption~\ref{assumpt:Fortin}; see~Proposition~\ref{prop:SV-op} for the existence of $\Pi_h$.
Now, to remove the constraint on the function in the infimum we proceed as in the proof of Theorem~\ref{thm:quasiopt-mixed}. For arbitrary $v_h \in X_h$ we consider
\begin{align}\label{def:zh-Pih}
	z_h \coloneqq v_h + \Pi_h(u - v_h), 
\end{align}
which satisfies $\tr(z_h) = \tr(g_h)$ and $\skp{\divergence z_h}{q_h} = 0$ for all $q_h\in Q_h$. In particular, $\divergence z_h\in Q_h$ implies $\divergence z_h \equiv 0$ in $L^2(\Omega)$ and 
\begin{align}\label{eq:ipm-first-p}
	\norm{p+\rho\divergence z_h}_{L^2(\Omega)} = \norm{p}_{L^2(\Omega)}.
\end{align}
Thus, using the stability of $\Pi_h$ according to Assumption~\ref{assumpt:Fortin}~\ref{itm:stab} we obtain
\begin{equation}\label{eq:ipm-first-u}
	\begin{aligned}
		\norm{\nabla u - \nabla z_h}_{L^2(\Omega)} 
		&\leq 
		\norm{\nabla u - \nabla v_h}_{L^2(\Omega)} 
		+ 	\norm{\nabla \Pi_h(u - v_h)}_{L^2(\Omega)} \\
		&	\leq (1+ c_{F}) 	\norm{\nabla u - \nabla v_h}_{L^2(\Omega)}
	\end{aligned}
\end{equation}
Combining~\eqref{eq:ipm-first-6} with~\eqref{eq:ipm-first-p} and~\eqref{eq:ipm-first-u} yields 
\begin{align}
	\norm{\nabla e_{h,1}}_{L^2(\Omega)}  \leq c_1 \inf_{\substack{v_h\in X_h}}  \norm{\nabla u -\nabla v_h}_{L^2(\Omega)} + c_2 \norm{p}_{L^2(\Omega)},
\end{align}
with $c_1 = 2(1+c_F)$ and $c_2 = \frac{1}{\sqrt{2 \rho \nu}} $, 
and hence shows the claim.
\end{proof}

\begin{remark}
	Taking the limit $\rho\to\infty$ in Lemma~\ref{lemma:quasi-opt_grad-div} recovers the quasi-optimal and pressure-robust error estimate in Theorem~\ref{thm:SV-inhom}.
\end{remark}

Now we have everything in place to prove an a~priori error estimate for the IPM, and the velocity estimate is asymptotically pressure-robust.

\begin{theorem}{(a priori error estimate for the IPM)}\label{thm:ipm-quasi-opt}
	Let~$\nu >0$, let $F \in H^{-1}(\Omega)^d$ and let~$g\in H^{1}_\sim(\Omega)$ be given. 
	Let $(u,p) \in \Hdiv \times \Lmeanfree(\Omega)$ be the weak solution to the Stokes problem~\eqref{eq:Stokes-weak-inh}.  
	For some $k \in \mathbb{N}$ let $X_h, X_{h,\sim,\divergence},V_h, Q_h$ be as in~\eqref{def:Xh-SV}, \eqref{def:Xhdiv-SV}, \eqref{def:Vh_SV} and~\eqref{def:Qh_SV} satisfying Assumptions~\ref{assumpt:FEM} with inf-sup constant $\overline \beta >0$ and Assumption~\ref{assumpt:pressure-sp}.
	Let $g_h \in X_h$ be such that  $\tr(g_h)=\tr(I_h g)= \tr(\Pi_h g)$, for the operators as in Proposition~\ref{prop:SV-op}. 
	Further, for $\rho>0$ and $\oldphi_{h,0}=0$ let $(u_{h,i})_{i \in \mathbb{N}}, (\oldphi_{h,i})_{i \in \mathbb{N}} \subset  X_{h,\sim}$, be sequences of iterate solutions to the IPM~\eqref{eq:ipm}. 
	
	Then, we have the following error estimates 
	\begin{align*}
		\norm{\nabla u - \nabla u_{h,i}}_{L^2(\Omega)} &\leq\, c_1(1+2\theta^{i-1}) \inf_{v_h \in X_{h}} \norm{\nabla u - \nabla v_h}_{L^2(\Omega)} + c_2\theta^{i-1} \norm{p}_{L^2(\Omega)},
		\\
		\norm{p - \divergence \oldphi_{h,i}}_{L^2(\Omega)} 
				&\leq\,
		c_1\left( \nu \overline \beta^{-1} + c_3 \theta^{i-1}\right) \inf_{v_h \in X_h} \norm{\nabla u - \nabla v_h}_{L^2(\Omega)} + 
		\inf_{q_h \in Q_h} \norm{p - q_h }_{L^2(\Omega)} \\
		& ~
		+ c_2 c_3 \theta^{i-1} \norm{p}_{L^2(\Omega)} ,
	\end{align*}
	where $\theta\coloneqq \frac{\nu}{\nu+\rho\overline \beta^2}$, $c_1 = 2(1 + c_F)$, 
	 $c_2= \frac{1}{\sqrt{2\rho\nu}}$, $c_3=\tfrac{\nu + \rho \overline \beta \sqrt{d} }{\overline \beta}$ and where $c_F>0$ is the constant in Assumption~\ref{assumpt:Fortin}~\ref{itm:stab}.
\end{theorem}

\begin{proof}
	Let $(u_h,p_h) \in X_{h,\sim,\divergence} \times Q_h$ be the discrete solution to the discrete mixed problem~\eqref{eq:Stokes-mixed-inh} with the Scott--Vogelius method.  
	Adding and subtracting $u_h$, we split the error into  $ u -  u_{h,i}$ and estimate the terms on the right-hand side of  
	\begin{align*}
		\norm{\nabla u - \nabla u_{h,i}}_{L^2(\Omega)} 
		&\leq\norm{\nabla u - \nabla u_h}_{L^2(\Omega)} + \norm{\nabla  u_h - \nabla u_{h,i}}_{L^2(\Omega)}
	\end{align*}
	separately. 
	By Lemma~\ref{lem:conv_ipm}, adding and subtracting $u$, and by Lemma~\ref{lemma:quasi-opt_grad-div} we have 
	\begin{align*}
		\norm{\nabla  u_h - \nabla u_{h,i}}_{L^2(\Omega)} & \leq \theta^{i-1} \norm{\nabla u_h - \nabla u_{h,1}}_{L^2(\Omega)} \\
		&\leq \theta^{i-1} \norm{\nabla u - \nabla u_h}_{L^2(\Omega)} + \theta^{i-1} \norm{\nabla u - \nabla u_{h,1}}_{L^2(\Omega)} \\
		&\leq \theta^{i-1} \norm{\nabla u - \nabla u_h}_{L^2(\Omega)} + \theta^{i-1} \left(c_1 \inf_{v_h \in X_h} \norm{\nabla u - \nabla v_h}_{L^2(\Omega)}  + c_2 \norm{p}_{L^2(\Omega)} \right)
	\end{align*}
    with constants $c_1=2(1+c_F), c_2=\tfrac{1}{\sqrt{2\rho\nu}}$.
	Proposition~\ref{thm:SV-inhom} yields 
	\begin{align*}
		\norm{\nabla u - \nabla u_h}_{L^2(\Omega)} \leq c_1 \inf_{v_h \in X_h} \norm{\nabla u - \nabla v_h}_{L^2(\Omega)}.
	\end{align*}
	Putting everything together we obtain
	\begin{align*}
		\norm{\nabla u - \nabla u_{h,i}}_{L^2(\Omega)} &\leq\, c_1(1+2\theta^{i-1}) \inf_{v_h \in X_{h}} \norm{\nabla u - \nabla v_h}_{L^2(\Omega)} + c_2\theta^{i-1} \norm{p}_{L^2(\Omega)}.
	\end{align*}
	Similarly, for the pressure error we use $p + \divergence \oldphi_{h,i} = p - p_h + p_h - \divergence \oldphi_{h,i}$ and estimate  
	\begin{align*}
		\norm{p- \divergence \oldphi_{h,i}}_{L^2(\Omega)} 
		\leq 	
		\norm{p- p_h }_{L^2(\Omega)}  + 	\norm{p_h- \divergence \oldphi_{h,i}}_{L^2(\Omega)} 		
	\end{align*}
	separately. For the first term we use Proposition~\ref{thm:SV-inhom}
	\begin{align*}
		\norm{p- p_h }_{L^2(\Omega)} \leq c_1 \nu (\overline{\beta})^{-1} \inf_{v_h \in X_h} \norm{\nabla u - \nabla v_h}_{L^2(\Omega)} + \inf_{q_h \in Q_h} \norm{p - q_h }_{L^2(\Omega)}.
	\end{align*}
	Then Lemma~\ref{lem:conv_ipm}  with $c_3\coloneqq\tfrac{\nu + \rho \overline \beta \sqrt{d} }{\overline \beta}$ and Lemma~\ref{lemma:quasi-opt_grad-div} yield
	\begin{align*}
		\norm{p_h- \divergence \oldphi_{h,i}}_{L^2(\Omega)} &\leq c_3 \theta^{i-1}\norm{\nabla u_h - \nabla u_{h,1}}_{L^2(\Omega)} \\
		&\leq c_3 \theta^{i-1} \left(c_1\inf_{v_h \in X_h} \norm{\nabla u - \nabla v_h}_{L^2(\Omega)} + c_2 \norm{p}_{L^2(\Omega)} \right).
	\end{align*}
	Combining both estimates shows 
	\begin{align*}
		\norm{p - \divergence \oldphi_{h,i}}_{L^2(\Omega)} 
		&\leq\,
		c_1\left( \nu \overline \beta^{-1} + c_3 \theta^{i-1}\right) \inf_{v_h \in X_h} \norm{\nabla u - \nabla v_h}_{L^2(\Omega)} + 
		\inf_{q_h \in Q_h} \norm{p - q_h }_{L^2(\Omega)} \\
		& ~
		+ c_2 c_3 \theta^{i-1} \norm{p}_{L^2(\Omega)} ,
	\end{align*}
 which proves the claim.
\end{proof}

\subsection{Monotonicity of the divergence}\label{sec:mon-div}

In addition to the exponential decay of the norm $\norm{\divergence u_{h,i}}_{L^2(\Omega)}$ as in Lemma~\ref{lem:conv_ipm}, monotonicity of the divergence norm can also be established. 
This is particularly simple by considering the corresponding energy, which is minimised in the $i$-th iteration of the IPM. 
More precisely, inserting the update~\eqref{eq:ipm-phi} for $i- 1$ into~\eqref{eq:ipm-u} we obtain that $u_{h,i} \in X_h$ such that $\tr(u_{h,i}) = \tr(g_h)$ solves 
\begin{align*}
	\nu \skp{\nabla u_{h,i}}{\nabla v_h} + \rho \skp{\divergence  u_{h,i}}{\divergence v_h} = \nu \skp{\nabla u_{h,i-1}}{\nabla v_h} \qquad \text{ for any } v_h \in V_h,
\end{align*}
for given $u_{h,i-1}\in X_h$ with $\tr(u_{h,i-1}) = \tr(g_h)$. We consider the energy 
\begin{align}\label{eq:def-Ji}
	{\mathcal{J}}_i(v_h) \coloneqq& \tfrac{\nu}{2} \norm{\nabla v_h}^2_{L^2(\Omega)} 
	+ \tfrac{\rho}{2}\norm{\divergence v_h}^2_{L^2(\Omega)}
	- \nu \skp{\nabla u_{h,i-1}}{\nabla v_h} \quad \text{for } v_h \in X_h. 
\end{align}
We may note that for the $i$th iterate of the IPM~\eqref{eq:ipm} we have
\begin{alignat*}{3}
	u_{h,i} & = \argmin_{v \in X_{h}\colon \tr(v) = \tr(g_h)} \mathcal{J}_i(v).
\end{alignat*}
 
This viewpoint allows us to prove monotonicity of the norm of the divergence in the iteration. 
\begin{theorem}\label{thm:mon}
	Under the conditions of Lemma~\ref{lem:conv_ipm} let $(u_{h,i})_{i \in \mathbb{N}}\subset X_h$ be the sequence of velocity solutions determined by the IPM~\eqref{eq:ipm}. 
	Then we have for any $i \geq 2$ that
	\begin{align*}
		\norm{\divergence u_{h,i}}_{L^2(\Omega)} \leq \norm{\divergence u_{h,i-1}}_{L^2(\Omega)}. 
	\end{align*}
\end{theorem}
\begin{proof}
Noting that $u_{h,i}$ is the minimiser of $\mathcal{J}_i(v_h)$ over all $v_h \in X_h$ such that $\tr(v_{h}) = \tr(g_h)$ we have in particular $\mathcal{J}_i(u_{h,i}) \leq \mathcal{J}_i(u_{h,i-1})$ and thus 
\begin{align*}
	\tfrac{\nu}{2}\norm{\nabla u_{h,i}}^2_{L^2(\Omega)} 
	+ \tfrac{\rho}{2}\norm{\divergence u_{h,i}}^2_{L^2(\Omega)} 
	&- \nu \skp{\nabla u_{h,i-1}}{\nabla u_{h,i}} \\
	&
	\leq -\tfrac{\nu}{2}\norm{\nabla u_{h,i-1}}^2_{L^2(\Omega)}
	+ \tfrac{\rho}{2}\norm{\divergence u_{h,i-1}}^2_{L^2(\Omega)}
\end{align*}
Rearranging terms yields 
\begin{align*}
0 \leq 	\tfrac{\nu}{2} \norm{\nabla u_{h,i} - \nabla u_{h,i-1}}^2_{L^2(\Omega)}
	\leq
	 \tfrac{\rho}{2} \left( \norm{\divergence u_{h,i-1}}^2_{L^2(\Omega)}  - \norm{\divergence u_{h,i}}^2_{L^2(\Omega)} \right),
\end{align*}
which proves the claim. 
\end{proof}

Notice that this property hinges on the fact that we can replace the pressure iterate \eqref{eq:ipm-phi} in the velocity equation \eqref{eq:ipm-u}, which is available because of $Q_h = \divergence V_h$, but not for general Uzawa methods of the form~\eqref{eq:Uzawa}.

\subsection{Numerical experiments for the iterated penalty method}\label{sec:ipm-num-exp}
In the following, we present some numerical tests in 2D on the iterated penalty method for the Scott--Vogelius element. 
We implemented the version presented in \cite[Alg. 7.2]{BernardiGiraultRaviartHechtRiviere2024}.

\subsubsection*{Manufactured solution and discretisation}
We use again the manufactured solution in~\eqref{eq:exact_sol} to the Stokes problem~\eqref{eq:Stokes-weak-inh} on $\Omega=(0,1)^2$ subject to inhomogeneous Dirichlet boundary conditions for given $g=u\in \Hsim$. 
The  triangulations for the discrete problem are generated using parameters $N\in\{4,8,16,32,64\}$ as before. 
As in Section~\ref{sec:NumExp_SV}, we directly impose the traces $\tr(u_{h,i}) = \tr(g_h)=\tr(\tilde{I}_h(g))$ in FEniCS with $\tilde I_h$, as introduced in Section~\ref{sec:compat}. 
Alternatively, by use of a discrete extension of $\tr(g_h)$, one may reduce the inhomogeneous problem~\eqref{eq:Uzawa-u} to the homogeneous one. 
However, in this case the convergence of the IPM depends on the choice of the discrete extension, as demonstrated in Subsection~\ref{subsec:dep_extension} below.
To ensure reproducibility, for this reason, in the remaining numerical experiments we impose the inhomogeneous Dirichlet boundary conditions directly by the FEniCS routine \texttt{DirichletBC}. 
We investigate the perfomance of the IPM on the original mesh $\tria_h$ compared to the modified mesh $\tilde{\tria}_h$ as described in Section~\ref{sec:mesh_mod}:  
\begin{enumerate}[label=$\bullet$]
	\item\label{itm:no-mod} 
	In $\tria_h$:  singular and nearly singular vertices may be present; 
	\item\label{itm:full-mod} In $\tilde{\tria}_h$: 
	local modification at all possibly singular vertices are employed so that no singular and no (nearly) singular vertices occur, see~Section~\ref{sec:mesh_mod}.
\end{enumerate}
All experiments are conducted with polynomial degree $k=4$, i.e., the velocity space is $X_h=\mathcal{L}_4^1(\tria_h)^2$ and $\mathcal{L}_4^1(\tilde{\tria}_h)^2$ respectively, see~\eqref{def:Xh-SV}. 
In the case of the modified mesh $\tilde{\tria}_h$ the pressure space is characterised as $Q_h = \mathcal{L}_3^0(\tilde{\tria}_h)\cap L_0^2(\Omega)$. 
For the original mesh $\tria_h$ due to the possible presence of singular vertices the pressure functions may have some local constraints. 
Since the IPM does not require a basis of the pressure space, we can approximate the solutions by using the IPM in this case without extra effort.

We use $\tol=10^{-11}$ for the termination criterion \eqref{eq:tol} of the IPM and we impose that at most $100$ iterations are performed. 
We investigate the penalisation parameter $\rho\in\{10^2, 10^4\}$. 

\subsubsection{Dependence on discrete extension}\label{subsec:dep_extension}

By use of an extension the Stokes problem~\eqref{eq:Stokes-weak-inh} subject to inhomogeneous boundary conditions can be reduced to one with homogeneous boundary conditions. 
We consider the manufactured solution $u$ in~\eqref{eq:exact_sol} to the Stokes problem~\eqref{eq:Stokes-weak-inh} on $\Omega=(0,1)^2$ and approximate the numerical solutions by using different discrete extensions. 

We choose $g\coloneqq u\in H^1_{\sim}(\Omega)$ the manufactured solution, and we consider a perturbed function $\tilde{g}= g + g_{\mathrm{perp}}$ for a divergence-free function $g_{\mathrm{perp}}$ with $g_{\mathrm{perp}}|_{\partial \Omega} = 0$. 
Thus, we have that $g|_{\partial \Omega} = \tilde g|_{\partial \Omega}$. 
Specifically, for a scaling parameter  $c_{\mathrm{perp}}\in\{10^3, 10^6\}$ and function $\chi\coloneqq \sin^2(\pi x_1)\sin^2(\pi x_2)$ we choose 
	\begin{align}\label{eq:uperp}
	g_{\mathrm{perp}} 
	= c_{\mathrm{perp}}
	\operatorname{curl} \chi
	= c_{\mathrm{perp}}
	 \begin{pmatrix}
		\pi\sin^2(\pi x_1)\sin(2\pi x_2) \\ -\pi \sin(2\pi x_1)\sin^2(\pi x_2)	\end{pmatrix}\in H^1_0(\Omega)^2\cap \Hdiv .
\end{align}
The discrete boundary datum is $\tilde{I}_h({g})|_{\partial \Omega} = \tilde{I}_h(\tilde{g})|_{\partial \Omega}$, with $\tilde{I}_h$ as in \eqref{def:tilde-Ih} and $\tilde{g}= g + g_{\mathrm{perp}}$. 
Then we consider $g_h = \tilde{I}_h(\tilde{g})$ as discrete extension and write the discrete solution to the inhomogeneous discrete problem as $u_{h,i}=\tilde{u}_{h,i}+ g_h$. 
Then, $\tilde{u}_{h,i}\in H^1_0(\Omega)^2$ can be computed by the IPM with homogeneous boundary conditions using this extension. 

 Table~\ref{tbl:IPM_hom_pert_1e3} shows the errors for the two choices of perturbation factor $c_{\mathrm{perp}}\in\{10^3, 10^6\}$. 
 In particular for mesh parameter $N=32$ one can observe different values of the $L^2$-error of the velocity and  of the $L^2$-norm of the divergence for the choices of scaling factor $c_{\mathrm{perp}}$. 
 Due to this dependency on the discrete extension and to avoid ambiguity, in the following we impose the inhomogeneous Dirichlet boundary conditions directly by the FEniCS routine \texttt{DirichletBC}.

\begin{table}[h]
	\small
	
	\begin{subtable}{\linewidth}
		\begin{tabular}{c rcccccc}
			\toprule
			& \;N & $h$\; & \; $i$ \;& ~$\norm{u - u_{h,i}}_{L^2(\Omega)}$ \;& \;$\norm{u - u_{h,i}}_{H^1(\Omega)}$ \;& \;$\norm{\divergence u_{h,i}}_{L^2(\Omega)}$\; & \;$\norm{p - p_{h,i}}_{L^2(\Omega)}$ \;\\
			\midrule
			\multirow{4}{*}{\rotatebox[origin=c]{90}{$c_{\mathrm{perp}}=10^3$}}
			& 4  & 0.29&  4 &  1.97e+00 & 1.03e+02  & 6.12e-13      &   2.01e+02 \\
			& 8 & 0.16  &   4  &  1.30e-01& 1.22e+01& 1.22e-12  & 3.59e+01\\
			& 16 & 0.08  &   3  &   2.96e-03 &  5.79e-01  & 6.12e-12  &    1.72e+00 \\
			& 32  & 0.04  &  3 & 1.68e-04 &   5.67e-02  & 4.58e-12   & 2.00e-01 \\
			& 64 & 0.02 &  3 &  1.27e-04&  3.00e-03  & 9.17e-12   &   9.13e-03 \\
			\bottomrule
		\end{tabular}
	\end{subtable}
	
	\vfill
	
	\begin{subtable}{\linewidth}
		\begin{tabular}{c rcccccc}
			\toprule
			& \;N & $h$\; & \; $i$ \;& ~$\norm{u - u_{h,i}}_{L^2(\Omega)}$ \;& \;$\norm{u - u_{h,i}}_{H^1(\Omega)}$ \;& \;$\norm{\divergence u_{h,i}}_{L^2(\Omega)}$\; & \;$\norm{p - p_{h,i}}_{L^2(\Omega)}$ \;\\
			\midrule
			\multirow{4}{*}{\rotatebox[origin=c]{90}{$c_{\mathrm{perp}}=10^6$}}
			& 4 &  0.29  &   4 &  1.97e+00 &  103.46  & 6.12e-13  & 201.33\\
			&  8 & 0.16 &  4 & 1.30e-01 &   12.23  & 1.22e-12   & 35.93\\
			&  16 & 0.08 &  3  & 8.43e-03 &  0.58  & 6.13e-12 &   1.72\\
			&  32&  0.04  & 3 &   2.98e-02  &0.23  & 4.58e-12  &  0.21\\
			& 64 & 0.02 &  3 & 1.27e-01 &   0.93 & 9.18e-12   &   0.26\\
			\bottomrule
		\end{tabular}
	\end{subtable}
	
	\caption{Error and divergence norms on modified mesh $\tilde{\tria}_h$ between IPM discrete solutions using the reduction to homogeneous boundary conditions for two different extensions of the boundary data, see \eqref{eq:uperp}, with parameter $c_{\mathrm{perp}}$.
		The experiment was performed with $k = 4$ and $g_h = \tilde I_h (\tilde g)$, and manufactured solution $(u,p)$ as in \eqref{eq:exact_sol} to the Stokes problem \eqref{eq:Stokes-weak-inh}, with $\Ra = 1$, penalty parameter $\rho=10^2$, tolerance $\tol = 1e-11$.}
	\label{tbl:IPM_hom_pert_1e3}
\end{table}

\subsubsection{Convergence on modified mesh $\tilde{\tria}_h$}
We investigate the error between the exact solution $(u,p)$ to the Stokes problem~\eqref{eq:Stokes-weak-inh} and the discrete IPM solution $(u_{h,i},p_{h,i})$ as in~\eqref{eq:ipm} for inhomogeneous Dirichlet boundary conditions. 
For penalty parameters $\rho=10^2$ and $\rho=10^4$ the corresponding results are contained in Table~\ref{tbl:IPM_rho2} and Table~\ref{tbl:IPM_rho4}, respectively. 
Recall that the contraction factor $\theta$ of the IPM decreases in the penalty parameter $\rho$, see Lemma~\ref{lem:conv_ipm}. 
Also in the numerical experiments we observe that the convergence rate increases in $\rho$; for mesh parameters $N\in\{4,8,16,32\}$, the norm of the approximate velocity divergence reaches the tolerance $\tol = 10^{-11}$ in only $3$ or $4$ iterations for $\rho=10^4$, see~Table~\ref{tbl:IPM_rho4}, whereas more than $14$ iterations are required for $\rho=10^2$, see~Table~\ref{tbl:IPM_rho2}. 
The $L^2$- and $H^1$-errors in the velocity and the $L^2$-error of the pressure error are almost identical for both values of $\rho$.

\input{fig/IPM_direct_H1u_errortable_rho2_4.tex}

\subsubsection{Effect of mesh modification}

\begin{figure}[h]
	\begin{subfigure}[b]{0.8\textwidth}
		\centering
		\begin{tikzpicture}[scale = 0.7]
			
			\definecolor{darkgray176}{RGB}{176,176,176}
			\definecolor{darkorange25512714}{RGB}{255,127,14}
			\definecolor{forestgreen4416044}{RGB}{44,160,44}
			\definecolor{lightgray204}{RGB}{204,204,204}
			\definecolor{steelblue31119180}{RGB}{31,119,180}
			
			\begin{axis}[
				legend cell align={left},
				legend style={fill opacity=0.8, draw opacity=1, text opacity=1, draw=black},
				legend pos=south east,
				log basis y={10},
				minor xtick={},
				minor ytick={},
				tick align=outside,
				tick pos=left,
				x grid style={darkgray176},
				xlabel={iteration $i$},
				xmajorgrids,
				xmin=-3.95, xmax=104.95,
				xtick style={color=black},
				xtick={-20,0,20,40,60,80,100,120},
				y grid style={darkgray176},
				ymajorgrids,
				ymin=1.57466661275396e-12, ymax=1e-2,
				ymode=log,
				ytick style={color=black},
				ytick={1e-14,1e-12,1e-10,1e-08,1e-06,0.0001,0.01,1,100},
				minor ytick={2e-11,3e-11,4e-11,5e-11,6e-11,7e-11,8e-11,9e-11,2e-10,3e-10,4e-10,5e-10,6e-10,7e-10,8e-10,9e-10,2e-09,3e-09,4e-09,5e-09,6e-09,7e-09,8e-09,9e-09,2e-08,3e-08,4e-08,5e-08,6e-08,7e-08,8e-08,9e-08,2e-07,3e-07,4e-07,5e-07,6e-07,7e-07,8e-07,9e-07,2e-06,3e-06,4e-06,5e-06,6e-06,7e-06,8e-06,9e-06,2e-05,3e-05,4e-05,5e-05,6e-05,7e-05,8e-05,9e-05,0.0002,0.0003,0.0004,0.0005,0.0006,0.0007,0.0008,0.0009,0.002,0.003,0.004,0.005,0.006,0.007,0.008,0.009,0.02,0.03,0.04,0.05,0.06,0.07,0.08,0.09,0.2,0.3,0.4,0.5,0.6,0.7,0.8,0.9,2,3,4,5,6,7,8,9,20,30,40,50,60,70,80,90}
				]
				
				\addplot [dashed,mark=none, thick] coordinates {
					(0,6.60e-11) (100,6.60e-11)
				};
				\addlegendentry{value for SV for $\tilde{\tria}_h$}
				
				\addplot [thick, steelblue31119180]
				table {%
					1	0.0024956901347373906
					2	0.0006805919750373422
					3	0.0003802516946913769
					4	0.00032955443219506784
					5	0.00030881326536124807
					6	0.0002929017539121079
					7	0.0002788971257828148
					8	0.00026611849630831896
					9	0.0002542473971551789
					10	0.00024309579814994355
					11	0.00023254563150818997
					12	0.00022251977416153969
					13	0.00021296550114337744
					14	0.0002038447719026942
					15	0.00019512850026320543
					16	0.00018679316696826472
					17	0.0001788188111353447
					18	0.00017118783395738194
					19	0.0001638842813624099
					20	0.00015689340950227648
					21	0.0001502014175577375
					22	0.0001437952799483927
					23	0.00013766263759195968
					24	0.00013179172461150502
					25	0.00012617131639447224
					26	0.00012079069065553793
					27	0.00011563959660538563
					28	0.00011070822921284723
					29	0.00010598720683679772
					30	0.0001014675511852974
					31	9.714066888116836e-05
					32	9.299833434875937e-05
					33	8.903267369038393e-05
					34	8.523614938485676e-05
					35	8.160154582837696e-05
					36	7.812195544710261e-05
					37	7.479076552826119e-05
					38	7.160164559081939e-05
					39	6.85485353070152e-05
					40	6.562563296570913e-05
					41	6.282738443992962e-05
					42	6.0148472586627066e-05
					43	5.758380713904221e-05
					44	5.512851503252547e-05
					45	5.277793111301125e-05
					46	5.052758927207153e-05
					47	4.8373213968517084e-05
					48	4.631071205621761e-05
					49	4.43361650187103e-05
					50	4.2445821547738596e-05
					51	4.063609033816132e-05
					52	3.8903533331907734e-05
					53	3.7244859123721225e-05
					54	3.5656916729021794e-05
					55	3.4136689597613e-05
					56	3.2681289852582306e-05
					57	3.128795282838008e-05
					58	2.995403178084788e-05
					59	2.8676992883216214e-05
					60	2.7454410403791556e-05
					61	2.628396207865495e-05
					62	2.5163424722068502e-05
					63	2.409066996807542e-05
					64	2.3063660241100882e-05
					65	2.208044490394724e-05
					66	2.1139156495696615e-05
					67	2.0238007254585948e-05
					68	1.9375285635261093e-05
					69	1.854935314312215e-05
					70	1.775864114362236e-05
					71	1.700164793827266e-05
					72	1.627693586648214e-05
					73	1.5583128615961858e-05
					74	1.4918908561648576e-05
					75	1.4283014297046305e-05
					76	1.3674238217987276e-05
					77	1.3091424242235923e-05
					78	1.2533465601684244e-05
					79	1.1999302733058473e-05
					80	1.1487921289266495e-05
					81	1.0998350176566396e-05
					82	1.0529659737324437e-05
					83	1.0080959994735546e-05
					84	9.65139890021396e-06
					85	9.24016079260887e-06
					86	8.84646479826911e-06
					87	8.469563364018738e-06
					88	8.108740819710202e-06
					89	7.76331206429882e-06
					90	7.432621189261886e-06
					91	7.116040318710242e-06
					92	6.812968352289859e-06
					93	6.522829840935111e-06
					94	6.2450739106035094e-06
					95	5.979173199487309e-06
					96	5.724622846424589e-06
					97	5.480939560030872e-06
					98	5.247660677558319e-06
					99	5.0243432982678094e-06
					100	4.8105634317639815e-06

				};
				\addlegendentry{no mod. $\tria_h$}
				
				\addplot [thick, forestgreen4416044]
				table {%
					1	0.0024385350090211154
					2	0.0005626733321674505
					3	0.00015014475721884034
					4	4.2254902769204267e-05
					5	1.2129394198746284e-05
					6	3.510223687856245e-06
					7	1.019636003980488e-06
					8	2.9674948715461803e-07
					9	8.646521292301065e-08
					10	2.521532217218645e-08
					11	7.358894192861247e-09
					12	2.1492345276340762e-09
					13	6.282035274170727e-10
					14	1.8379463054062066e-10
					15	5.387782738401704e-11
					16	1.5960991401375284e-11
					17	5.221423859648737e-12

				};
				\addlegendentry{mod. $\tilde{\tria}_h$}
			\end{axis}
			
		\end{tikzpicture}	
		\caption{Plot of \centering$\norm{\divergence u_{h,i}}_{L^2(\Omega)}$ over the iteration. The dashed line represents the value of the corresponding Scott--Vogelius solution $(u_h,p_h)$ to~\eqref{eq:Stokes-mixed-inh} for $\tilde{\tria}_h$. 
		}
		\label{fig:IPM-iter_div}
	\end{subfigure}\\
	\vspace{0.5cm}
	\begin{subfigure}[b]{0.39\textwidth}
		\includegraphics[width=\linewidth]{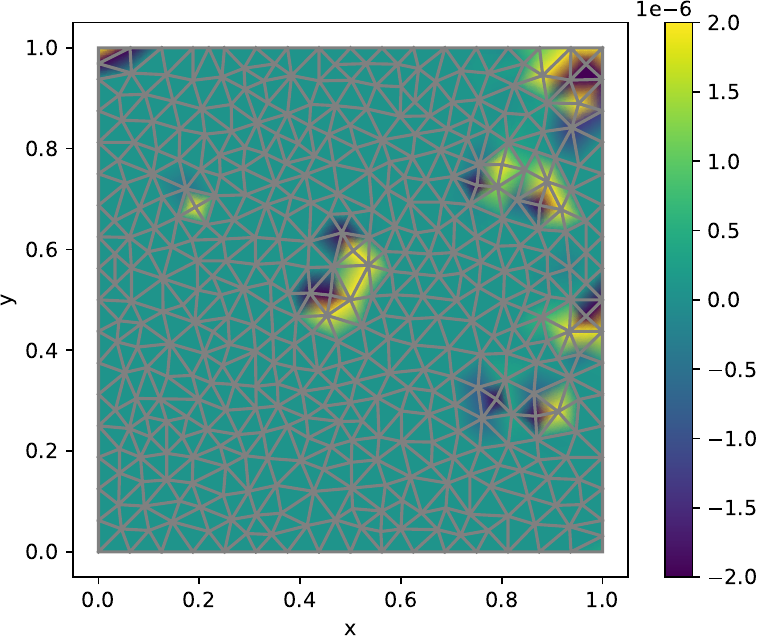}
		\caption{\centering $\divergence u_{h,17}$ for $\tria_h$, \linebreak
			\qquad $\norm{\divergence u_{h,17}}_{\ell^\infty(\Omega)} =$2.86e-02}
		\label{fig:IPM-mesh-nomod}
	\end{subfigure}
	\hspace{1cm}
	\begin{subfigure}[b]{0.39\textwidth}
		\includegraphics[width=\linewidth]{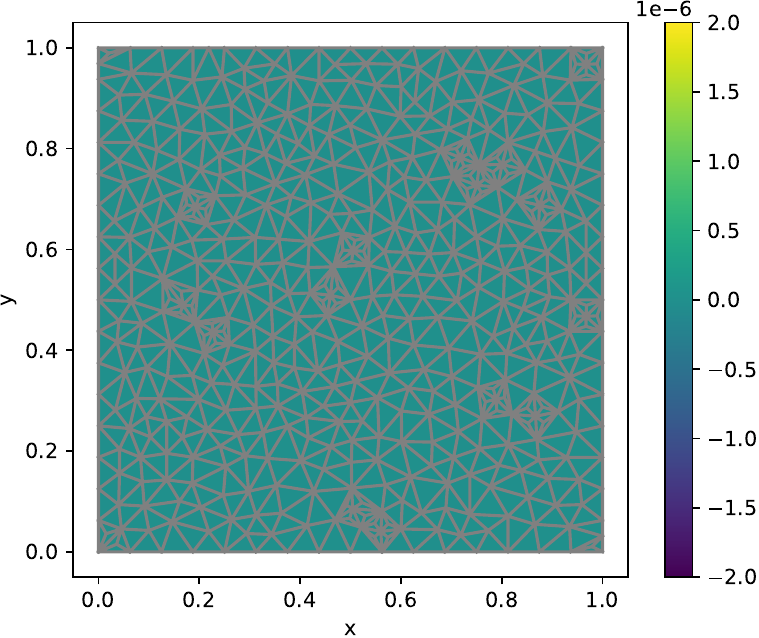}
		\caption{\centering $\divergence u_{h,17}$ for $\tilde{\tria}_h$, \linebreak 
			\qquad 
			$\norm{\divergence u_{h,17}}_{\ell^\infty(\Omega)} =$ 5.48e-10} 
		\label{fig:IPM-mesh-full}
	\end{subfigure}
	\caption{Comparison between the original mesh $\tria_h$ and the modified mesh $\tilde{\tria}_h$ (cf. Sec.~\ref{sec:mesh_mod}) for $N=16$: divergence norm of the IPM discrete velocity solution to~\eqref{eq:ipm} with $k=4$, $g_h=\tilde{I}_h (g)$ with $\Ra = 1$, penalty parameter $\rho=10^2$, tolerance $\tol =$ 1e-11 and data of the manufactured solution~\eqref{eq:exact_sol} to the Stokes problem~\eqref{eq:Stokes-weak-inh}.
	} 
	\label{fig:IPM-mesh}
\end{figure}
Here, we showcase the effect of nearly singular vertices in the mesh on the convergence of the IPM. 

Only on the modified mesh $\tilde{\tria}_h$
the norm of the divergence of the velocity iterates reaches the tolerance, see Figure~\ref{fig:IPM-mesh}~(\subref{fig:IPM-iter_div}). On the original mesh $\tria_h$ the convergence of the divergence errors  is much slower. This is likely due to smaller inf-sup constant at nearly singular vertices.
Independently of the type of mesh in Figure~\ref{fig:IPM-mesh}~(\subref{fig:IPM-iter_div}) we may observe that the divergence norm decays monotonically in the number of iterations, which illustrates the monotonicity result in Section~\ref{sec:mon-div}. 

\input{fig/IPM-BGHRR_direct_pressure-robustness_rho10.tex}

Furthermore, we investigate $\divergence u_{h,i}$ locally depending on the mesh modification. 
In Figure~\ref{fig:IPM-mesh}~(\subref{fig:IPM-mesh-nomod}) and (\subref{fig:IPM-mesh-full}) we compare the divergence of the velocity iterates for the IPM after $17$ iterations for the original mesh $\tria_h$ (left) and modified mesh $\tilde{\tria}_h$ (right) for $N=16$, plotted over the domain. 
One can see that on the original mesh the divergence is large mainly in the corners and at nearly singular vertices, see Figure~\ref{fig:IPM-mesh}~(\subref{fig:IPM-mesh-nomod}).
Moreover, $\norm{\divergence{u_{h,17}}}_{\ell^\infty}$ is  several orders of magnitude smaller for the modified mesh $\tilde{\tria}_h$,  see Figure~\ref{fig:IPM-mesh}~(\subref{fig:IPM-mesh-full}) than on the original one. 
The discrete norm $\norm{\cdot}_{\ell^\infty}$ is computed as the maximal absolute value in every dof of the velocity divergence projected onto the space of discontinuous piecewise polynomials of degree $k-1$.

\subsubsection{Pressure robustness}
Finally, we visualise the asymptotic pressure robustness property of the iterated penalty method for the Scott--Vogelius element. 
As in Section~\ref{sec:press-rob}, we follow~\cite[Example 1.1]{John2017} and we consider $\Ra\in\{10^{11}, 10^{12}, 10^{13}\}$. 
Recall that the given function $F$ in~\eqref{eq:ipm} depends on this parameter. 
Figure~\ref{fig:IPM-pressure_robustness} shows the error norms after $5,10, 15$ and $20$ iterations, respectively.

We observe that the scaling of the velocity error with $\Ra$ decreases with increasing number of iterations. For $i=15$ and $i=20$ iterations its qualitative behaviour is similar to the error of the Scott--Vogelius solution, see Fig.~\ref{fig:SV_TH_direct_pressure_robustness_loglog}.

\subsection*{Acknowledgments.} 
We owe thanks to Charles Parker for useful discussions and to Leo Rebholz for valuable suggestions. We also thank the referees for their helpful comments.

The work by T.T.~was supported by  the German Research Foundation (DFG) via grant TRR 154, subproject C09, project number 239904186 and by the Hans and Ria Messer Foundation. 

\FloatBarrier

\printbibliography

\end{document}

%% file: fig/H1u_L2p_errorplot_Ras_SV_TH_direct_loglog.tex
\begin{figure}
	\begin{subfigure}[b]{0.39\textwidth}
		\begin{tikzpicture}[scale=0.75]
			
			\definecolor{crimson2143940}{RGB}{214,39,40}
			\definecolor{darkgray176}{RGB}{176,176,176}
			\definecolor{darkorange25512714}{RGB}{255,127,14}
			\definecolor{darkturquoise23190207}{RGB}{23,190,207}
			\definecolor{forestgreen4416044}{RGB}{44,160,44}
			\definecolor{goldenrod18818934}{RGB}{188,189,34}
			\definecolor{gray127}{RGB}{127,127,127}
			\definecolor{lightgray204}{RGB}{204,204,204}
			\definecolor{mediumpurple148103189}{RGB}{148,103,189}
			\definecolor{orchid227119194}{RGB}{227,119,194}
			\definecolor{sienna1408675}{RGB}{140,86,75}
			\definecolor{steelblue31119180}{RGB}{31,119,180}
			
			\begin{axis}[
				name=mainplot,
				log basis x={2},
				log basis y={10},
				tick align=outside,
				tick pos=left,
				x grid style={darkgray176},
				xlabel={N},
				xmajorgrids,
				xmin = 4,
				xmax = 64,
				xmode=log,
				xtick style={color=black},
				xtick={4,8,16,32,64},
				y grid style={darkgray176},
				ymajorgrids,
				ymin=0.0011698876090081, ymax=2e+04,
				ymode=log,
				ytick style={color=black},
				ytick={0.0001,0.001,0.01,0.1,1,10,100,1000,10000,100000,1000000,10000000}
				]
				\addplot [thick, forestgreen4416044]
				table {%
					4 103.459360895077
					8 12.2262294884181
					16 0.578630735408128
					32 0.056693321245978
					64 0.00285469236380621
				};
				\addplot [thick, forestgreen4416044, dash pattern=on 1pt off 3pt on 3pt off 3pt]
				table {%
					4 257.102978680948
					8 37.7178110775521
					16 1.87352388530219
					32 0.205477874184573
					64 0.00961764993914605
				};
				\addplot [thick, mediumpurple148103189]
				table {%
					4 103.459589928574
					8 12.2262243744914
					16 0.578630323990959
					32 0.0569796854833243
					64 0.00286818739580242
				};
				\addplot [thick, mediumpurple148103189, dash pattern=on 1pt off 3pt on 3pt off 3pt]
				table {%
					4 1611.67548578056
					8 120.30059775616
					16 7.57964558314424
					32 0.521666470077306
					64 0.0316501104456925
				};
				\addplot [thick, darkorange25512714]
				table {%
					4 103.461976556163
					8 12.2261883831231
					16 0.578636982595379
					32 0.0743662645281142
					64 0.00374072139071592
				};
				\addplot [thick, darkorange25512714, dash pattern=on 1pt off 3pt on 3pt off 3pt]
				table {%
					4 15991.7810431442
					8 1148.66486135905
					16 73.8345759476428
					32 4.81488521746653
					64 0.303058981680196
				};
			\end{axis}
			\input{legend_pressure-robustness-SV.tex}
		\end{tikzpicture}

		\caption{\centering Scott--Vogelius}
		
	\end{subfigure}
	\hfill
	\begin{subfigure}[b]{0.39\textwidth}
		\begin{tikzpicture}[scale=0.75]
			
			\definecolor{crimson2143940}{RGB}{214,39,40}
			\definecolor{darkgray176}{RGB}{176,176,176}
			\definecolor{darkorange25512714}{RGB}{255,127,14}
			\definecolor{darkturquoise23190207}{RGB}{23,190,207}
			\definecolor{forestgreen4416044}{RGB}{44,160,44}
			\definecolor{goldenrod18818934}{RGB}{188,189,34}
			\definecolor{gray127}{RGB}{127,127,127}
			\definecolor{lightgray204}{RGB}{204,204,204}
			\definecolor{mediumpurple148103189}{RGB}{148,103,189}
			\definecolor{orchid227119194}{RGB}{227,119,194}
			\definecolor{sienna1408675}{RGB}{140,86,75}
			\definecolor{steelblue31119180}{RGB}{31,119,180}
			
			\begin{axis}[
				legend cell align={left},
				legend style={fill opacity=0.8, draw opacity=1, text opacity=1, draw=lightgray204},
				log basis x={2},
				log basis y={10},
				tick align=outside,
				tick pos=left,
				x grid style={darkgray176},
				xlabel={N},
				xmajorgrids,
				xmin = 4,
				xmax = 64,
				xmode=log,
				xtick style={color=black},
				xtick={4,8,16,32,64},
				y grid style={darkgray176},
				ymajorgrids,
				ymin=0.00119464184872399, ymax=2e+04,
				ymode=log,
				ytick style={color=black},
				ytick={0.0001,0.001,0.01,0.1,1,10,100,1000,10000,100000,1000000,10000000}
				]
				\addplot [thick, forestgreen4416044]
				table {%
					4 132.238016134766
					8 11.0023583004681
					16 0.688727871623153
					32 0.049180983434019
					64 0.00296231370660114
				};
				\addplot [thick, forestgreen4416044, dash pattern=on 1pt off 3pt on 3pt off 3pt]
				table {%
					4 230.188680432167
					8 16.5400227773952
					16 1.08140643579353
					32 0.0706061080534323
					64 0.00444880538943585
				};
				\addplot [thick, mediumpurple148103189]
				table {%
					4 1160.41529447223
					8 86.4690278325173
					16 5.7951009610878
					32 0.380284747834067
					64 0.0242106306147037
				};
				\addplot [thick, mediumpurple148103189, dash pattern=on 1pt off 3pt on 3pt off 3pt]
				table {%
					4 2288.47920094577
					8 164.187849634634
					16 10.7817668641559
					32 0.70386462314792
					64 0.0443886550530257
				};
				\addplot [thick, darkorange25512714]
				table {%
					4 11589.4803210306
					8 862.00085435514
					16 57.8312574836036
					32 3.790031353455
					64 0.241504052230258
				};
				\addplot [thick, darkorange25512714, dash pattern=on 1pt off 3pt on 3pt off 3pt]
				table {%
					4 22882.3237792477
					8 1641.66639343867
					16 107.812914626928
					32 7.03841430068553
					64 0.443875376269867
				};

			\end{axis}
		\end{tikzpicture}

		\caption{\centering Taylor--Hood}
	\end{subfigure}
	\hfill
	\caption{For different values $\Ra\in\{10^{11}, 10^{12}, 10^{13}\}$ plot of the error norms for Scott--Vogelius and Taylor--Hood discrete solutions $(u_h, p_h)$ on the modified mesh $\tilde{\tria}_h$, cf. Section~\ref{sec:mesh_mod}, respectively, to \eqref{eq:Stokes-mixed-inh} with $k=4$, $g_h=\tilde{I}_h g$ and manufactured solution $(u,p)$~\eqref{eq:exact_sol} to the Stokes problem \eqref{eq:Stokes}.}
	\label{fig:SV_TH_direct_pressure_robustness_loglog}
\end{figure}

%% file: legend_pressure-robustness-SV.tex
	\definecolor{crimson2143940}{RGB}{214,39,40}
\definecolor{darkgray176}{RGB}{176,176,176}
\definecolor{darkorange25512714}{RGB}{255,127,14}
\definecolor{darkturquoise23190207}{RGB}{23,190,207}
\definecolor{forestgreen4416044}{RGB}{44,160,44}
\definecolor{goldenrod18818934}{RGB}{188,189,34}
\definecolor{gray127}{RGB}{127,127,127}
\definecolor{lightgray204}{RGB}{204,204,204}
\definecolor{mediumpurple148103189}{RGB}{148,103,189}
\definecolor{orchid227119194}{RGB}{227,119,194}
\definecolor{sienna1408675}{RGB}{140,86,75}
\definecolor{steelblue31119180}{RGB}{31,119,180}
\begin{axis}[
	at={(mainplot.outer north east)}, 
	anchor=north west,
	xshift=3.5cm, 
	yshift=-1cm,
	hide axis,
	scale only axis,
	width=0pt,
	height=0pt,
	xmin=0, xmax=1, ymin=0, ymax=1,
	legend columns=1,
	legend style={
		font=\small,
		row sep=2pt,
		draw=lightgray204,
		draw = none,
	}
	]
	
	\addplot[black, thick] coordinates {(0,0)};
	\addlegendentry{$\norm{\nabla(u-u_h)}_{L^2(\Omega)}$}
	
	\addplot[black, dash pattern=on 1pt off 3pt on 3pt off 3pt, thick] coordinates {(0,0)};
	\addlegendentry{$\norm{p-p_h}_{L^2(\Omega)}$}
	
	\addplot[forestgreen4416044, thick] coordinates {(0,0)};
	\addlegendentry{Ra = $10^{11}$}
	
	\addplot[mediumpurple148103189, thick] coordinates {(0,0)};
	\addlegendentry{Ra = $10^{12}$}
	
	\addplot[darkorange25512714, thick] coordinates {(0,0)};
	\addlegendentry{Ra = $10^{13}$}

\end{axis}

%% file: fig/IPM_direct_H1u_errortable_rho2_4.tex
\begin{table} \small
	\begin{tabular}{rcccccc} 
		\toprule
		\;N & $h$\; & \; $i$ \;& ~$\norm{u - u_{h,i}}_{L^2(\Omega)}$ \;& \;$\norm{u - u_{h,i}}_{H^1(\Omega)}$ \;& \;$\norm{\divergence u_{h,i}}_{L^2(\Omega)}$\; & \;$\norm{p - p_{h,i}}_{L^2(\Omega)}$ \;\\
		\midrule
		4 & 0.29  &  16   &  1.97e+00 &    1.03e+02  & 8.60e-12   &      2.01e+02\\
		8 & 0.16   & 20 &    1.30e-01  &   1.22e+01  & 5.80e-12    &     3.59e+01 \\
		16 & 0.08  &  17    & 2.96e-03   &  5.79e-01  & 5.22e-12     &    1.72e+00\\
		32 & 0.04   & 14 &    1.66e-04   &  5.67e-02  & 7.62e-12  &       2.00e-01\\
		64 & 0.02 &   14  &   3.80e-06 &   2.85e-03  & 9.30e-12     &    9.13e-03\\
		\bottomrule
	\end{tabular}
	\caption{Error and divergence norms on modified mesh $\tilde{\tria}_h$ between IPM discrete solutions $(u_{h,i}, p_{h,i})$ to \eqref{eq:ipm} with $k = 4$ and $g_h = \tilde I_h (g)$, and manufactured solution $(u,p)$ as in \eqref{eq:exact_sol} to the Stokes problem \eqref{eq:Stokes-weak-inh}, with $\Ra = 1$, penalty parameter $\rho=10^2$, tolerance $\tol =$ 1e-11. }
	\label{tbl:IPM_rho2}
\end{table}

\begin{table} \small 
	\begin{tabular}{rcccccc}
		\toprule
		\;N & $h$\; & \; $i$ \; & \;$\norm{u - u_{h,i}}_{L^2(\Omega)}$ \;& \;$\norm{u - u_{h,i}}_{H^1(\Omega)}$ \;& \;$\norm{\divergence u_{h,i}}_{L^2(\Omega)}$\; & \;$\norm{p - p_{h,i}}_{L^2(\Omega)}$ \;\\
		\midrule
		4 & 0.29 &    4  &   1.97e+00  &   1.03e+02  & 6.02e-13    &     2.01e+02\\
		8 & 0.16 &    4  &   1.30e-01  &   1.22e+01  & 1.22e-12     &   3.59e+01\\
		16&  0.08  &   3   &  2.96e-03  &   5.79e-01  & 6.11e-12    &     1.72e+00\\
		32 & 0.04   &  3  &   1.66e-04  &   5.67e-02  & 4.58e-12   &      2.00e-01\\
		64&  0.02 &    3   &  3.81e-06  &   2.85e-03  & 9.17e-12    &     9.13e-03\\
		\bottomrule
	\end{tabular}
	\caption{Error and divergence norms on modified mesh $\tilde{\tria}_h$ between IPM discrete solutions $(u_{h,i}, p_{h,i})$ to \eqref{eq:ipm} with $k = 4$ and $g_h = \tilde I_h (g)$, and manufactured solution $(u,p)$ as in \eqref{eq:exact_sol} to the Stokes problem \eqref{eq:Stokes-weak-inh}, with $\Ra = 1$, penalty parameter $\rho=10^4$, tolerance $\tol=$ 1e-11. }
	\label{tbl:IPM_rho4}
\end{table}

%% file: fig/IPM-BGHRR_direct_pressure-robustness_rho10.tex
\begin{figure}
	\begin{subfigure}[b]{0.48\textwidth}
		\begin{tikzpicture}[scale = 0.8]
			
			\definecolor{crimson2143940}{RGB}{214,39,40}
			\definecolor{darkgray176}{RGB}{176,176,176}
			\definecolor{darkorange25512714}{RGB}{255,127,14}
			\definecolor{darkturquoise23190207}{RGB}{23,190,207}
			\definecolor{forestgreen4416044}{RGB}{44,160,44}
			\definecolor{goldenrod18818934}{RGB}{188,189,34}
			\definecolor{gray127}{RGB}{127,127,127}
			\definecolor{lightgray204}{RGB}{204,204,204}
			\definecolor{mediumpurple148103189}{RGB}{148,103,189}
			\definecolor{orchid227119194}{RGB}{227,119,194}
			\definecolor{sienna1408675}{RGB}{140,86,75}
			\definecolor{steelblue31119180}{RGB}{31,119,180}
			
			\begin{axis}[
				log basis x={2},
				log basis y={10},
				tick align=outside,
				tick pos=left,
				x grid style={darkgray176},
				xlabel={N},
				xmajorgrids,
				xmin = 4,
				xmax = 64,
				xmode=log,
				xtick style={color=black},
				xtick={4,8,16,32,64},
				y grid style={darkgray176},
				ymajorgrids,
				ymin=5, ymax=3e+05,
				ymode=log,
				ytick style={color=black},
				minor ytick=8
				]
				\addplot [thick, forestgreen4416044]
				table {%
					4	104.12864535213453
					8	13.243668712205967
					16	101.7220977963988
					32	6.175173765971724
					64	5.692315656482406
					
				};
				\addplot [thick, forestgreen4416044, dash pattern=on 1pt off 3pt on 3pt off 3pt]
				table {%
					4	449.70573910096016
					8	283.6125932033863
					16	2327.4541792885257
					32	267.59397830105576
					64	262.3085283328427

				};
				\addplot [thick, mediumpurple148103189]
				table {%
					4	163.4301123389638
					8	73.98272098934805
					16	1017.2587835178678
					32	61.819331969008154
					64	56.924670692811226
					
				};
				\addplot [thick, mediumpurple148103189, dash pattern=on 1pt off 3pt on 3pt off 3pt]
				table {%
					4	4050.8666373648234
					8	2874.553841687339
					16	23274.83608376449
					32	2676.142466320852
					64	2623.0895514839362
					
				};
				\addplot [thick, darkorange25512714]
				table {%
					4	1277.6417987660836
					8	748.242486825227
					16	10172.640301777614
					32	618.2631888222933
					64	569.2482273485784
					
				};
				\addplot [thick, darkorange25512714, dash pattern=on 1pt off 3pt on 3pt off 3pt]
				table {%
					4	40483.643644185766
					8	28803.645464915913
					16	232748.6607899032
					32	26761.627920033116
					64	26230.89980525149
					
				};
				
			\end{axis}
		\end{tikzpicture}
		\caption{$i=5$}
	\end{subfigure}
	\hfill
	\begin{subfigure}[b]{0.48\textwidth}
		\begin{tikzpicture}[scale = 0.8]
			
			\definecolor{crimson2143940}{RGB}{214,39,40}
			\definecolor{darkgray176}{RGB}{176,176,176}
			\definecolor{darkorange25512714}{RGB}{255,127,14}
			\definecolor{darkturquoise23190207}{RGB}{23,190,207}
			\definecolor{forestgreen4416044}{RGB}{44,160,44}
			\definecolor{goldenrod18818934}{RGB}{188,189,34}
			\definecolor{gray127}{RGB}{127,127,127}
			\definecolor{lightgray204}{RGB}{204,204,204}
			\definecolor{mediumpurple148103189}{RGB}{148,103,189}
			\definecolor{orchid227119194}{RGB}{227,119,194}
			\definecolor{sienna1408675}{RGB}{140,86,75}
			\definecolor{steelblue31119180}{RGB}{31,119,180}
			
			\begin{axis}[
				name=mainplot,
				log basis x={2},
				log basis y={10},
				tick align=outside,
				tick pos=left,
				x grid style={darkgray176},
				xlabel={N},
				xmajorgrids,
				xmin = 4,
				xmax = 64,
				xmode=log,
				xtick style={color=black},
				xtick={4,8,16,32,64},
				y grid style={darkgray176},
				ymajorgrids,
				ymin=2e-03, ymax=1e+05,
				ymode=log,
				ytick style={color=black},
				minor ytick=8
				]
				\addplot [thick, forestgreen4416044]
				table {%
					4	103.45922106954893
					8	12.225586678322582
					16	0.6063026051244889
					32	0.05638108584440866
					64	0.002866517504522458
					
				};
				\addplot [thick, forestgreen4416044, dash pattern=on 1pt off 3pt on 3pt off 3pt]
				table {%
					4	257.095979431302
					8	37.688956945007305
					16	4.416839126884497
					32	0.19827352544280355
					64	0.025898017929060328
					
				};
				\addplot [thick, mediumpurple148103189]
				table {%
					4	103.45822554385192
					8	12.219900476145071
					16	1.9110836320453857
					32	0.0565352903146271
					64	0.012073720673072344
					
				};
				\addplot [thick, mediumpurple148103189, dash pattern=on 1pt off 3pt on 3pt off 3pt]
				table {%
					4	1611.6654822785044
					8	120.21497133181936
					16	40.76174114017578
					32	0.6513931210520866
					64	0.2763242450208421
					
				};
				\addplot [thick, darkorange25512714]
				table {%
					4	103.45166725803682
					8	12.172571460252614
					16	18.23355212364681
					32	0.1900001987951612
					64	0.12279076426281461
					
				};
				\addplot [thick, darkorange25512714, dash pattern=on 1pt off 3pt on 3pt off 3pt]
				table {%
					4	15991.782673798056
					8	1148.6253394208866
					16	407.3090860204666
					32	6.548531263659993
					64	2.7933040698345115
					
				};
				
			\end{axis}	
		\end{tikzpicture}
		\caption{$i=10$}
	\end{subfigure}
	\vspace{0.5cm}
	\begin{subfigure}[b]{0.48\textwidth}
		\begin{tikzpicture}[scale = 0.8]
			
			\definecolor{crimson2143940}{RGB}{214,39,40}
			\definecolor{darkgray176}{RGB}{176,176,176}
			\definecolor{darkorange25512714}{RGB}{255,127,14}
			\definecolor{darkturquoise23190207}{RGB}{23,190,207}
			\definecolor{forestgreen4416044}{RGB}{44,160,44}
			\definecolor{goldenrod18818934}{RGB}{188,189,34}
			\definecolor{gray127}{RGB}{127,127,127}
			\definecolor{lightgray204}{RGB}{204,204,204}
			\definecolor{mediumpurple148103189}{RGB}{148,103,189}
			\definecolor{orchid227119194}{RGB}{227,119,194}
			\definecolor{sienna1408675}{RGB}{140,86,75}
			\definecolor{steelblue31119180}{RGB}{31,119,180}
			
			\begin{axis}[
				name=mainplot,
				log basis x={2},
				log basis y={10},
				minor y tick num=8,
				yminorgrids=true,
				tick align=outside,
				tick pos=left,
				x grid style={darkgray176},
				xlabel={N},
				xmajorgrids,
				xmin = 4,
				xmax = 64,
				xmode=log,
				xtick style={color=black},
				xtick={4,8,16,32,64},
				y grid style={darkgray176},
				ymajorgrids,
				ymin=1e-03, ymax=1e+05,
				ymode=log,
				ytick style={color=black},
				]
				\addplot [thick, forestgreen4416044]
				table {%
					4	103.45936035425878
					8	12.226228014506079
					16	0.5786314866535172
					32	0.05668870485955504
					64	0.002854099535535819
					
				};
				\addplot [thick, forestgreen4416044, dash pattern=on 1pt off 3pt on 3pt off 3pt]
				table {%
					4	257.1029723253309
					8	37.71774989098306
					16	1.873578504577052
					32	0.20536927328110297
					64	0.00960680615755121
					
				};
				\addplot [thick, mediumpurple148103189]
				table {%
					4	103.45958528497754
					8	12.226214972014427
					16	0.5786520134366597
					32	0.05668760607538639
					64	0.0028568173897607053
					
				};
				\addplot [thick, mediumpurple148103189, dash pattern=on 1pt off 3pt on 3pt off 3pt]
				table {%
					4	1611.675474698849
					8	120.30041019144089
					16	7.580276494435266
					32	0.5205564968486065
					64	0.03167614676943136
					
				};
				\addplot [thick, darkorange25512714]
				table {%
					4	103.46196149443941
					8	12.226084842166728
					16	0.5802042053483727
					32	0.05669902137559302
					64	0.003338496661402021
					
				};
				\addplot [thick, darkorange25512714, dash pattern=on 1pt off 3pt on 3pt off 3pt]
				table {%
					4	15991.781030081624
					8	1148.66466493972
					16	73.84031730297951
					32	4.811670995744008
					64	0.3037492318149506
					
				};
				
			\end{axis}		
		\end{tikzpicture}
		\caption{$i=15$}
	\end{subfigure}
	\hfill
	\begin{subfigure}[b]{0.48\textwidth}
		\begin{tikzpicture}[scale = 0.8]
			
			\definecolor{crimson2143940}{RGB}{214,39,40}
			\definecolor{darkgray176}{RGB}{176,176,176}
			\definecolor{darkorange25512714}{RGB}{255,127,14}
			\definecolor{darkturquoise23190207}{RGB}{23,190,207}
			\definecolor{forestgreen4416044}{RGB}{44,160,44}
			\definecolor{goldenrod18818934}{RGB}{188,189,34}
			\definecolor{gray127}{RGB}{127,127,127}
			\definecolor{lightgray204}{RGB}{204,204,204}
			\definecolor{mediumpurple148103189}{RGB}{148,103,189}
			\definecolor{orchid227119194}{RGB}{227,119,194}
			\definecolor{sienna1408675}{RGB}{140,86,75}
			\definecolor{steelblue31119180}{RGB}{31,119,180}
			
			\begin{axis}[
				name=mainplot,
				log basis x={2},
				log basis y={10},
				minor y tick num=8,
				yminorgrids=true,
				tick align=outside,
				tick pos=left,
				x grid style={darkgray176},
				xlabel={N},
				xmajorgrids,
				xmin = 4,
				xmax = 64,
				xmode=log,
				xtick style={color=black},
				xtick={4,8,16,32,64},
				y grid style={darkgray176},
				ymajorgrids,
				ymin=1e-03, ymax=1e+05,
				ymode=log,
				ytick style={color=black},
				]
				\addplot [thick, forestgreen4416044]
				table {%
					4	103.45936047282197
					8	12.226229292695235
					16	0.5786307222084194
					32	0.056688854169150725
					64	0.002854408629623222
					
				};
				\addplot [thick, forestgreen4416044, dash pattern=on 1pt off 3pt on 3pt off 3pt]
				table {%
					4	257.10297870808773
					8	37.717811316485495
					16	1.8735254944006767
					32	0.20537440218557979
					64	0.009621988600080155
					
				};
				\addplot [thick, mediumpurple148103189]
				table {%
					4	103.4595864705444
					8	12.226227744211798
					16	0.5786307643449078
					32	0.05668909640781493
					64	0.002859763487011093
					
				};
				\addplot [thick, mediumpurple148103189, dash pattern=on 1pt off 3pt on 3pt off 3pt]
				table {%
					4	1611.6754838714353
					8	120.30060259237587
					16	7.579648730933958
					32	0.5205782940687939
					64	0.031793863544917464
					
				};
				\addplot [thick, darkorange25512714]
				table {%
					4	103.46197334779151
					8	12.226212473632987
					16	0.5786334127297473
					32	0.05671364341619817
					64	0.003351667158345421
					
				};
				\addplot [thick, darkorange25512714, dash pattern=on 1pt off 3pt on 3pt off 3pt]
				table {%
					4	15991.781039218358
					8	1148.6648658667373
					16	73.83457703732246
					32	4.811701630338171
					64	0.3044556705595459
					
				};
				
			\end{axis}		
		\end{tikzpicture}
		\caption{$i=20$}
	\end{subfigure}	
	
	\begin{tikzpicture}[scale = 0.75]
		\input{legend_pressure-robustness-IPM.tex}
	\end{tikzpicture}
	
	\caption{	
		For different values $\Ra\in\{10^{11}, 10^{12}, 10^{13}\}$ plot of the error norms for IPM discrete solutions $(u_{h,i},p_{h,i})$ to \eqref{eq:ipm} after $i\in\{5, 10, 15, 20\}$ iterations on the modified mesh $\tilde{\tria}_h$, respectively, with $k=4$, $g_h=\tilde{I}_h g$, $\mathrm{tol}=10^{-11}$, $\rho=10^2$ and manufactured solution $(u,p)$ as in \eqref{eq:exact_sol} to the Stokes problem \eqref{eq:Stokes}.
		Note that for $i=15$ and $i = 20$ lines are not visible because they lie on top of each other.}
	\label{fig:IPM-pressure_robustness} 
\end{figure}

%% file: legend_pressure-robustness-IPM.tex
\definecolor{crimson2143940}{RGB}{214,39,40}
\definecolor{darkgray176}{RGB}{176,176,176}
\definecolor{darkorange25512714}{RGB}{255,127,14}
\definecolor{darkturquoise23190207}{RGB}{23,190,207}
\definecolor{forestgreen4416044}{RGB}{44,160,44}
\definecolor{goldenrod18818934}{RGB}{188,189,34}
\definecolor{gray127}{RGB}{127,127,127}
\definecolor{lightgray204}{RGB}{204,204,204}
\definecolor{mediumpurple148103189}{RGB}{148,103,189}
\definecolor{orchid227119194}{RGB}{227,119,194}
\definecolor{sienna1408675}{RGB}{140,86,75}
\definecolor{steelblue31119180}{RGB}{31,119,180}
\begin{axis}[
	hide axis,
	xmin=0, xmax=1,
	ymin=0, ymax=1,
	width=0pt,
	height=0pt,
	scale only axis,
	legend columns=2,
	legend style={
		at={(0.5,1)},
		anchor=north,
		draw=lightgray204,
		font=\small,
		row sep=3pt,
		column sep=1cm
	}
	]
	
	\addplot[black, thick] coordinates {(0,0)};
	\addlegendentry{$\|\nabla(u - u_{h,i})\|_{L^2(\Omega)}$}
	
	\addplot[black, dash pattern=on 1pt off 3pt on 3pt off 3pt, thick] coordinates {(0,0)};
	\addlegendentry{$\|p - p_{h,i}\|_{L^2(\Omega)}$}
	
	\addplot[forestgreen4416044, thick] coordinates {(0,0)};
	\addlegendentry{Ra = $10^{11}$}
	
	\addplot[mediumpurple148103189, thick] coordinates {(0,0)};
	\addlegendentry{Ra = $10^{12}$}
	
	\addplot[darkorange25512714, thick] coordinates {(0,0)};
	\addlegendentry{Ra = $10^{13}$}
	
\end{axis}